\documentclass[12pt,reqno]{amsart}
\usepackage{eucal}
\usepackage{enumerate,amsfonts,mathrsfs,amscd,amssymb,amsthm,amsmath,bm,graphicx,psfrag,subfigure,xypic}

\numberwithin{equation}{section}

\newtheorem{lem}{Lemma}[section]
\newtheorem{cor}[lem]{Corollary}
\newtheorem{prop}[lem]{Proposition}
\newtheorem{thm}[lem]{Theorem}

\theoremstyle{definition}
\newtheorem{defn}[lem]{Definition}

\renewcommand\ell{l}
\newcommand\supp{\operatorname{supp}}

\newcommand \tf{\tilde f}
\newcommand \tW{\widetilde W}

\newcommand\End{\operatorname{End}}

\begin{document}

\title{Resolution of Indecomposable Integral Flows\\ on Signed Graphs}

\author{Beifang Chen}
\address{Department of Mathematics, Hong Kong University of Science and
Technology, Clear Water Bay, Kowloon, Hong Kong}
\email{mabfchen@ust.hk}

\author{Jue Wang}
\address{Department of Mathematics and Physics, Shenzhen Polytechnic,
Shenzhen, Guangdong Province, 518088, P.R. China}
\email{twojade@alumni.ust.hk}

\author{Thomas Zaslavsky}
\address{Department of Mathematical Sciences, Binghamton University
(SUNY), Binghamton, NY 13902-6000, U.S.A.}
\email{zaslav@math.binghamton.edu}

\begin{abstract}
It is well known that each nonnegative integral flow on a graph can
be decomposed into a sum of nonnegative graphic circuit flows, which
cannot be further decomposed into nonnegative integral sub-flows.
This is equivalent to saying that the indecomposable flows on graphs
are those graphic circuit flows. Turning from graphs to signed
graphs, the indecomposable flows are much richer than those of
unsigned graphs. This paper gives a complete description of
indecomposable flows on signed graphs from the viewpoint of
resolution of singularities by means of double covering graphs.
\end{abstract}

\subjclass[2000]{Primary 05C22; Secondary 05C20, 05C21, 05C38}

\keywords{Signed graph, double covering graph, sesqui-Eulerian
signed graph, prime sesqui-Eulerian signed graph, sesqui-Eulerian
circle-tree, indecomposable integral flow}

\date{\today}

\thanks{The research of the first author was supported by RGC Competitive Earmarked
Research Grants 600608, 600409, and 600811.}

\maketitle


\section{Introduction}

A {\em signed graph} is a graph in which each edge is given either a
positive or a negative sign.
A real or integral flow (or ``circulation'') on an ordinary unsigned
graph is a real- or integer-valued function on (oriented)
edges such that the net inflow to each vertex is zero. Analogously,
a {\em real flow} on a signed graph is a real-valued function on
(oriented) signed edges such that the net inflow to each vertex is
zero, and an \emph{integral flow}, a concept introduced by
Bouchet \cite{Bouchet}, is a flow whose values are integers.
The theory of flows on ordinary graphs is the specialization of the
signed-graph theory to the case in which all edges are positive.

There are many reasons to be interested in integral flows on graphs;
important ones are their connection to integer programming through
network optimization and their relationship to graph structure through the
analysis of  {\em conformally indecomposable flows}, that
is, integral flows that cannot be decomposed as the sum of two
integral flows whose flow values have the same sign on each edge
(both $\geq 0$ or both $\leq 0$). It is well known, and an important
observation in the theory of integral network flows, that the
indecomposable flows are identical to the \emph{circuit flows},
which are flows on circuits of the graphic matroid and which (in a suitable orientation of the graph) have value $1$ on the edges of a graph circuit (i.e., a connected 2-regular subgraph) and $0$ elsewhere.
The extension of the theory of indecomposable integral flows to signed graphs by Chen and Wang
\cite{Chen-Wang3}, carried out by an algorithmic method, led to the remarkable
discovery that, besides the anticipated circuit flows, which are already more
complicated in signed graphs than in ordinary graphs, there are many
``strange'' indecomposable flows with
elaborate self-intersection structure not describable by circuits of the signed graph.
Chen and Wang obtained a classification of indecomposable flows by means of their algorithm.

The present paper, by contrast, has a structural approach.
We characterize indecomposable flows by the method
of covering graphs, lifting each vertex and each edge of a signed
graph to two vertices and two edges of a double covering graph.
The strange indecomposable flows are regarded as singular phenomena, which we resolve by lifting them (blowing up repeated vertices and overlapping edges) to ordinary circuit flows in the covering graph.
In comparison to the algorithmic approach in \cite{Chen-Wang3}, the present paper hints at a connection (at least conceptually) between graph theory and resolution of singularities through covering spaces, as we think of lifting as a combinatorial analogue of resolution of singularities in algebraic geometry.
We believe this connection may be useful for studying gain graphs \cite{BG1}, which are more complicated than signed graphs.

The resolution process leads, via ``sesqui-Eulerian circle-trees" (Definition \ref{cycle-tree-defn}), to the following half-integral conformal decomposition (see Theorem \ref{Half-Integer Scale Decomposition}(d)):
\emph{Every nonzero integral flow on a signed graph can be conformally decomposed into a
half-integral positive linear combination of signed-graphic circuit flows.}
(For this purpose the half-integers include the integers.)

\section{Graphs and Signed Graphs}

\subsection{Graphs}\

A {\em graph} $G$ is a pair consisting of a {\em vertex set} $V(G)$ and
an {\em edge set} $E(G)$, such that each edge $x\in E(G)$ is associated
with a multiset $\End(x)$ of two vertices, called the endpoints of $x$.
(We use letters $e, x, y, z$ for edges.)
The edge $x$ is called a {\em link} if the two vertices of $\End(x)$ are
distinct and is called a {\em loop} if the two vertices are identical.
Let $x$ be an edge with $\End(x)=\{u,v\}$; we say that $x$ is {\em incident} with
$u$ and $v$, and that $u$ and $v$ are {\em adjacent} by $x$.

An incidence can be treated as an object, denoted by $(u,x)$ (which we think of as the \emph{end} of $x$ at $u$); thus the edge $x$ has the two incidences, or \emph{ends}, $(u,x)$ and $(v,x)$.  Although a loop $x$ (where $u=v$) has the same endpoint twice, we treat its two ends as two distinct objects; this is necessary in order to define orientation of signed edges.  When occasionally the notation must distinguish the two ends of a loop we write them as $(u,x), (v,x)$ as if $x$ were a link.  For a vertex $v$ we define
\[
\End(v) = \{\text{edge ends $(v,e)$ incident with } v\};
\]
in particular a loop at $v$ has two distinct ends in $\End(v)$.

A {\em cut-edge} is an edge whose deletion increases the number of connected components.
A \emph{cut-vertex} is a vertex whose deletion, together with all incident edges, increases the
number of components, or that is incident with a loop and at least one other edge.
A \emph{circle} (also known as a circuit, cycle, or polygon) is a subgraph that is connected and regular of degree $2$, or the edge set of such a subgraph.  The collection of circles of a graph $G$ forms the circuit system of a matroid, known as the {\em graphic matroid}, on the edge set of $G$.

A \emph{block} is a maximal connected subgraph without cut-vertices.  Thus, loops, cut-edges, and isolated vertices are blocks.  We call blocks \emph{adjacent} if they have a common vertex (which is necessarily a cut-vertex).  An \emph{end block} is a block adjacent to exactly one other block.  A \emph{circle block} is a block that is a circle.

A \emph{walk} of length $n$ in a graph is a sequence of vertices and edges,
\[
W = v_0e_1v_1e_2\cdots v_{n-1}e_nv_n,
\]
such that $\End(x_i)=\{u_{i-1},u_i\}$.  The {\em initial vertex} is $v_0$, $v_n$ is the
{\em terminal vertex}, and $v_1,\ldots,v_{n-1}$ are {\em internal vertices} of $W$.
We write
\[
W^{-1} = v_ne_nv_{n-1}\cdots e_2v_1e_1v_0
\]
for the same walk in the reverse direction; its initial vertex is $v_n$ and its terminal vertex is $v_0$.
A {\em subwalk} of $W$ is a subsequence of the form $v_ie_{i+1}v_{i+1}e_{i+2} \cdots v_{k-1}e_kv_k$.
A walk is \emph{closed} if $n\geq 1$ and $v_0=v_n$ and \emph{open} otherwise.  A walk is
a {\em trail} if it has no repeating edges, and
an {\em open path} if it has no repeating vertices (and consequently no repeating edges), and a {\em closed path} if it has no repeating vertices except that $v_0=v_n$.
The graph of a closed path is a circle.  (The difference is that a closed path has an initial and terminal vertex and a direction.)
Note that every walk, including a closed walk, has an initial vertex and a terminal vertex, which are identical if the walk is closed.

\subsection{Signed Graphs}\

A \emph{signed graph} $\Sigma=(G,\sigma)$ consists of an unsigned graph $G$ together
with a \emph{sign function} $\sigma: E(G)\to \{+1,-1\}$. We usually
write $V(G)$ and $E(G)$ as $V(\Sigma)$ and $E(\Sigma)$, respectively.
Subgraphs of $\Sigma$ inherit the edge signs from $\Sigma$.

The sign of a walk $W=v_0e_1v_1e_2\cdots v_{n-1}e_nv_n$ is the product
\[
\sigma(W)=\prod_{i=1}^n\sigma(e_i).
\]
In particular, the sign of a circle is the product of the signs of its edges.
A subgraph (or its edge set) is \emph{balanced} if every circle in it has positive sign.

A \emph{signed-graph circuit} is a subgraph (or its edge set) of the following three types:
\begin{enumerate}[\hspace{2ex} (1)]
\item
A positive circle, said to be of {\em Type I}.
\item
A pair of negative circles whose intersection is a single vertex, said to be of {\em Type II}.
\item
A pair of vertex-disjoint negative circles together with a path of positive length that connects the two circles and is
internally disjoint from the two circles, said to be of {\em Type III}.
\end{enumerate}
The connecting path in Type III and the common vertex in Type II are called the {\em circuit path} of the circuit.

The circuits of a signed graph $\Sigma$ form the circuit system of a matroid on the edge set of $\Sigma$ \cite{SG}, the \emph{frame matroid} of the signed graph $\Sigma$; such a matroid is called a {\em signed-graphic matroid}.
(Type II and Type III circuits are named {contrabalanced tight handcuffs} and {contrabalanced loose handcuffs} respectively by Zaslavsky \cite{SG}. We do not use these names here.)
An ordinary unsigned graph is viewed as a signed graph whose edges are all positive; so all of its circles have positive sign and the frame matroid of an unsigned graph coincides with the graphic matroid.

\subsection{Orientation}\

A \emph{bidirection} of a graph (a concept introduced by
Edmonds~\cite{MMP}) is a function $\omega$ from the set of all edge
ends to the sign group, $\{-1,+1\}$. We view a positive value
$\omega(u,e)$ as denoting an arrow at the end $(u,e)$ directed along
the edge $e$ toward the endpoint $u$, and a negative value as an
arrow directed away from the endpoint.  Recall that we treat a loop
$e=uv$ (with $u=v$) as having distinguishable ends $(u,e)$ and
$(v,e)$.

An \emph{orientation} of a signed graph $\Sigma$ \cite{OSG} is a bidirection
$\omega$ on its underlying graph such that for each edge $e$, with endpoints $u$ and $v$,
\begin{equation*}
\sigma(e) = -\omega(u,e)\omega(v,e).
\end{equation*}
So a positive edge $e$ must have two arrows in the same direction along $e$, indicating a direction of $e$ as in an ordinary directed graph.
A negative edge $e$ has two opposite arrows, which both point toward or both point away from its endpoints.
We let $(\Sigma,\omega)$ denote an oriented signed graph throughout.

A \emph{sink} in $(\Sigma,\omega)$ is a vertex $v$ at which all
edges point toward $v$, that is, $\omega(v,e)=+1$ for all edges $e$
at $v$. Conversely, a \emph{source} is a vertex $v$ at which all
edges at $v$ point away from $v$, that is, $\omega(v,e)=-1$ for all edges $e$ at $v$.

Two oriented edges $e_1,e_2\in E$ having a common endpoint $v$,
with the orientations $\omega(v,e_1)$ and $\omega(v,e_2)$, are {\em coherent at $v$} if
\begin{equation*}
\omega(v,e_1)\omega(v,e_2)=-1, \quad \text{i.e.,}\quad
\omega(v,e_1)+\omega(v,e_2)=0.
\end{equation*}
This means that $e_1$ and $e_2$ have a common direction (locally) at
their common endpoint $v$. A walk $W=v_0e_1v_1e_2 \cdots
v_{n-1}e_nv_n$ in $(\Sigma,\omega)$ is \emph{coherent at $v_i$} when
$e_{i}$ and $e_{i+1}$ are coherent at $v_i$; that is, when
\[
\omega(v_i,e_i)\omega(v_i,e_{i+1})=-1,\quad \text{i.e.,}
\quad
\omega(v_i,e_i)+\omega(v_i,e_{i+1})=0;
\]
and similarly a closed walk is {\em coherent at $v_0$} $(=v_n)$ when
\[
\omega(v_0,e_1)\omega(v_n,e_n)=-1, \quad \text{i.e.,}
\quad
\omega(v_0,e_1)+\omega(v_n,e_n)=0.
\]
The walk $W$ is a {\em coherent walk} when it is coherent at its internal vertices and, if it is closed, also coherent at the initial and terminal vertex.
(Coherence is meaningless at the initial and terminal vertices of an open walk.)
A simple fact is:

\begin{lem}\label{L:walksign}
In an oriented signed graph, the sign of a walk of positive length
equals $(-1)^{\ell+1} \omega(v_0,e_1)\omega(v_n,e_n)$,
where $\ell$ is the number of incoherent internal vertices.
If the walk is closed, its sign equals $(-1)^k$, where
$k$ is the total number of times that the walk is incoherent at vertices,
including the initial and terminal vertex (which is counted as a single vertex).
\end{lem}
\begin{proof}
We perform a short calculation. Let $W=v_0e_1v_1\cdots e_nv_n$ be the walk.  Its sign is
\begin{align*}
\sigma(W) &= \prod_{i=1}^n \sigma(e_i)
= \prod_{i=1}^n \big(-\omega(v_{i-1},e_i)\omega(v_{i},e_i)\big) \\
&= -\omega(v_0,e_1)\omega(v_n,e_n) \prod_{i=1}^{n-1}
\big(-\omega(v_{i},e_i)\omega(v_{i},e_{i+1})\big) \\
&=  -(-1)^{\ell}\omega(v_0,e_1)\omega(v_n,e_n)  \\
& = (-1)^k \quad \text{ if $W$ is closed.}
\qedhere
\end{align*}
\end{proof}

A \emph{direction} of $W$ is an assignment $\omega_W$ to each edge $e_i$ in $W$ of an orientation that is coherent at all internal vertices.  (The walk orientation $\omega_W$ is separate from the orientation $\omega$ of $\Sigma$.)
Every walk of positive length has exactly two directions, opposite to each other. A {\em directed walk} $(W,\omega_W)$ is a walk $W$ with a direction $\omega_W$.

Let $S\subseteq E$ be an edge set. The {\em reorientation} of $\Sigma$ by $S$
is the orientation $\omega_S$ obtained from $\omega$ by
reversing the orientations of the edges in $S$ and keeping the
orientations of edges outside $S$ unchanged. Thus $\omega_S$ is
given by
\[
\omega_S(v,e) =\begin{cases} -\omega(v,e) &\text{ if } e \in S, \\
\phantom{-} \omega(v,e) &\text{ if } e \notin S. \end{cases}
\]

Let $\omega_i$ be orientations on subgraphs $\Sigma_i$ of
$\Sigma$, $i=1,2$. The {\em coupling} of $\omega_1$ and $\omega_2$
is a function $[\omega_1,\omega_2]: E(\Sigma)\to\{-1,0,+1\}$, defined for each edge $e$ (having endpoint $v$) by
\begin{equation*}
[\omega_1,\omega_2](e)=\begin{cases}
\phantom{-}1 & \text{if } e\in E(\Sigma_1)\cap E(\Sigma_2)$, $\omega_1(v,e)=\omega_2(v,e), \\
-1 & \text{if } e\in E(\Sigma_1)\cap E(\Sigma_2)$, $\omega_1(v,e)\neq \omega_2(v,e), \\
\phantom{-} 0 & \text{otherwise.}
\end{cases}
\end{equation*}
(The definition is independent of which endpoint $v$ is.)
One may extend $\omega_i$ to $\Sigma$ by requiring $\omega_i(v,e)=0$
whenever the edge $e$ is not incident with the vertex $v$ in
$\Sigma_i$. We always assume this extension automatically. Then
alternatively,
\[
[\omega_1,\omega_2](e) =\omega_1(v,e)\,\omega_2(v,e).
\]

\subsection{The Double Covering Graph}\

\subsubsection{Defining the covering graph}

The {\em double covering graph} of $\Sigma$ is an unsigned graph
$\widetilde\Sigma$ whose vertex and edge sets are
\[
V(\widetilde\Sigma)=V(\Sigma)\times\{+,-\} \quad\text{and}\quad
E(\widetilde\Sigma)=E(\Sigma)\times \{+,-\},
\]
with adjacency defined as follows:  If vertices
$u,v\in V(\Sigma)$ are adjacent by an edge $e\in E(\Sigma)$, then the vertices
$(u,\alpha)$ and $(v,\alpha\,\sigma(e))$ in $V(\widetilde\Sigma)$
are adjacent by an edge in $E(\widetilde\Sigma)$, and the
vertices $(u,-\alpha)$ and $(v,-\alpha\,\sigma(x))$ in
$V(\widetilde\Sigma)$ are adjacent by another edge in $E(\widetilde\Sigma)$.
We denote these two edges by $\tilde e$ and $\tilde e^*$ (the signs on edges in $E(\widetilde\Sigma)$ are not edge signs; they are a notational convenience to ensure that $E(\widetilde\Sigma)$ contains two copies of each edge of $\Sigma$).
For simplicity, we write $u^\alpha=(u,\alpha)$.
When $e$ is a negative loop at its unique endpoint $v$, the edges
$\tilde e$ and $\tilde e^*$ are two parallel edges in $\widetilde\Sigma$ with
the endpoints $v^{+}$ and $v^-$.

We may think of $V(\widetilde\Sigma)$ as having two levels:
\begin{align*}
V^+ &= \{v^+:v\in V(\Sigma)\}\ \text{(the positive level)}, \\
V^- &= \{v^-:v\in V(\Sigma)\}\ \text{(the negative level)}.
\end{align*}
A positive edge is lifted to two edges, one inside the positive level and
the other inside the negative level; a negative edge is
lifted to two edges crossing between the two levels.
It is therefore impossible to lift all edges of
$\Sigma$ to the same level when $\Sigma$ is unbalanced.

The asterisk marks a canonical involutory, fixed-point-free
graph automorphism ${}^*$ of $\widetilde\Sigma$, defined by
\[
(v^\alpha)^* = v^{-\alpha}, \hspace{2ex} (\tilde e)^* = \tilde e^*, \hspace{2ex} (\tilde e^*)^* = \tilde e.
\]
There is also a canonical graph homomorphism
$\pi:\widetilde\Sigma\to\Sigma$, called the {\em projection}
of $\widetilde\Sigma$ to $\Sigma$, which is a pair of functions
$\pi_V:V(\widetilde\Sigma)\to V(\Sigma)$ and
$\pi_E:E(\widetilde\Sigma)\to E(\Sigma)$,
defined respectively by
\[
\pi_V(v^\alpha)=v \quad\text{and}\quad \pi_E(\tilde e)=e.
\]
Usually, we write $\pi_V$ and $\pi_E$ simply as $\pi$.

Often it is convenient to denote $\tilde e$ and $\tilde e^*$ by $e^\beta$ and $e^{-\beta}$, respectively, for some (arbitrary) choice of $\beta\in\{+,-\}$.  If $u$ and $v$ are the endpoints of $e$, then
\[
e^\beta = u^\alpha v^{\alpha\,\sigma(e)}, \quad e^{-\beta}=u^{-\alpha}v^{-\alpha\,\sigma(e)}, \quad \text{ for some } \alpha\in\{+,-\}.
\]
In the symbol $e^\beta$, $\beta$ is not related to the sign of the
edge $e$ in $\Sigma$. (The choice of whether $\tilde e$ is called
$e^+$ or $e^-$ does not change the double covering graph; it is only
a choice of names for edges.)

\subsubsection{Orienting the double covering graph}
An orientation $\omega$ on $\Sigma$ lifts to an orientation
$\widetilde\omega$ on $\widetilde\Sigma$, called the {\em lift} of
$\omega$. Let $e\in E(\Sigma)$ be an edge incident with
a vertex $v\in V(\Sigma)$, and let $e$ be lifted to an edge
$e^\beta\in E(\widetilde\Sigma)$ incident with a vertex $v^\alpha\in V(\widetilde\Sigma)$.
Define
\begin{equation}\label{Lift-Orientation}
\widetilde\omega(v^\alpha,e^\beta) = \alpha\,\omega(v,e).
\end{equation}
Since $v^{-\alpha}$ is an endpoint of the lifted edge $e^{-\beta}$, then by definition
\[
\widetilde\omega(v^{-\alpha},e^{-\beta}) =-\alpha\,\omega(v,e).
\]
The two arrows on each lifted edge $e^\beta$
are directed the same way along the edge, regardless of the sign $\sigma(e)$. In
fact, for each edge $e$ with endpoints $u$ and $v$ (possibly $u=v$), we have
\[
\widetilde\omega(u^\alpha,e^\beta)\,
\widetilde\omega(v^{\alpha\,\sigma(e)},e^\beta)
=\alpha\,\omega(u,e)\,\alpha\,\sigma(e)\,\omega(v,e) =-1.
\]
This means that the two arrows on $e^\beta$ have the same direction, as in
Figure~\ref{Lift-Orientation-Link} for a link and
Figure~\ref{Lift-Orientation-Loop} for a loop. So
$(\widetilde\Sigma,\widetilde\omega)$ is an ordinary oriented unsigned graph.
\begin{figure}[h]
\centering
\subfigure[$\sigma(e)=1$]{\includegraphics[width=25mm]{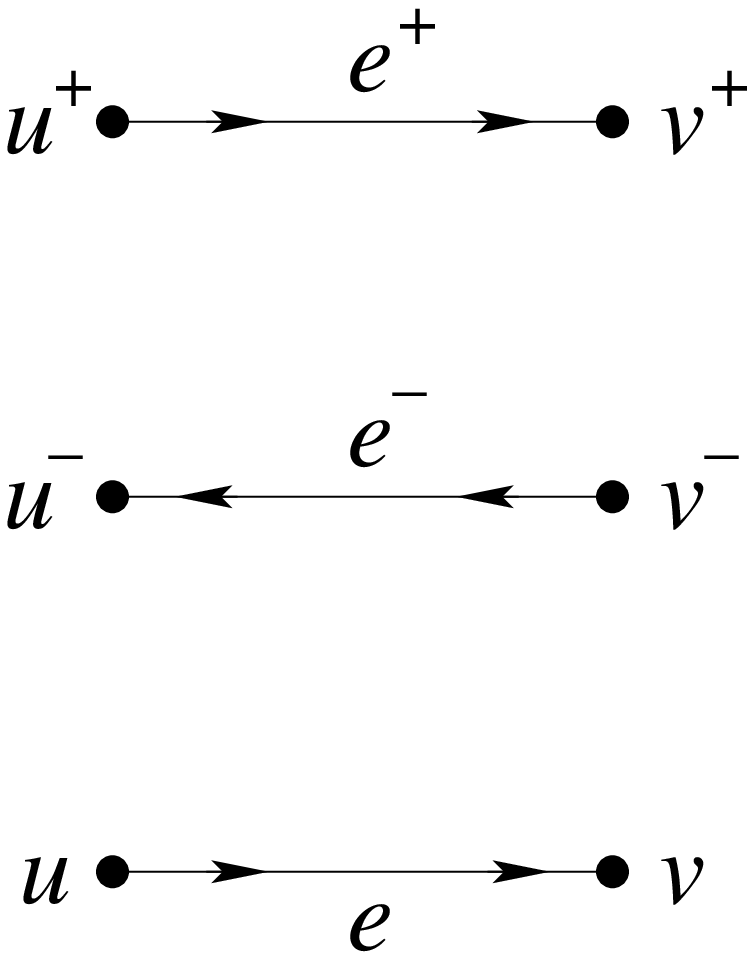}}
\hspace{10mm}
\subfigure[$\sigma(e)=-1$]{\includegraphics[width=25mm]{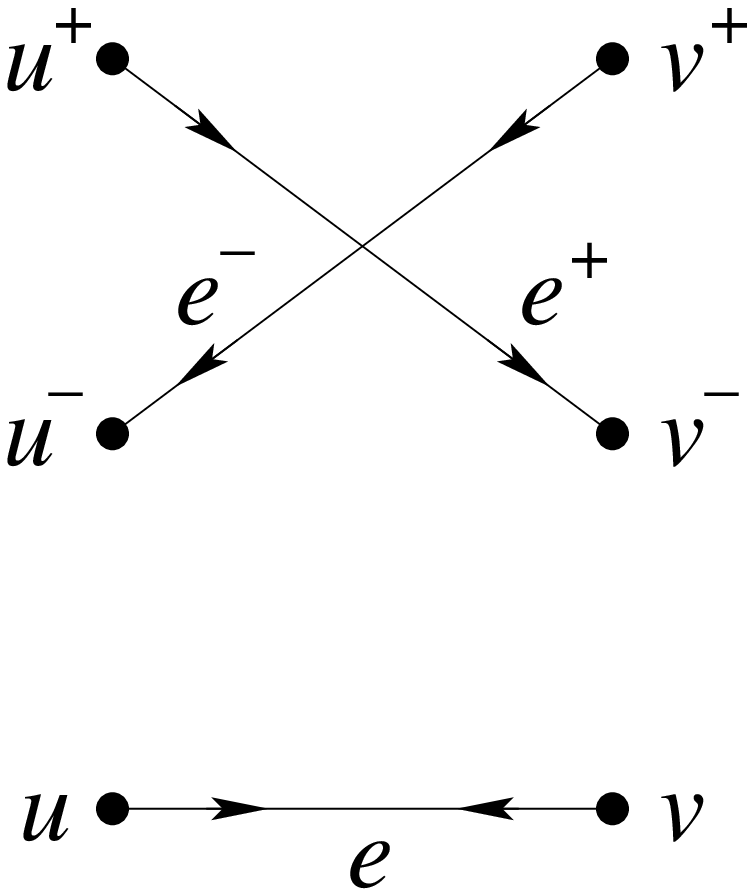}}
\hspace{10mm}
\subfigure[$\sigma(e)=-1$]{\includegraphics[width=25mm]{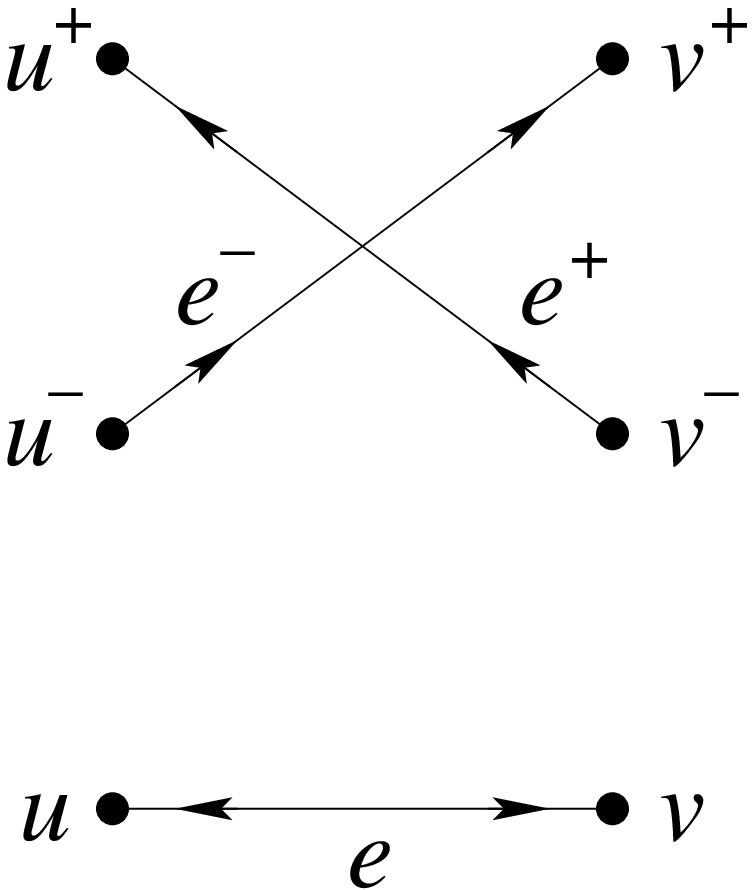}}
\caption{Lifting of a link and its orientation.}\label{Lift-Orientation-Link}
\end{figure}
\begin{figure}[h]
\centering
\includegraphics[width=12.8mm]{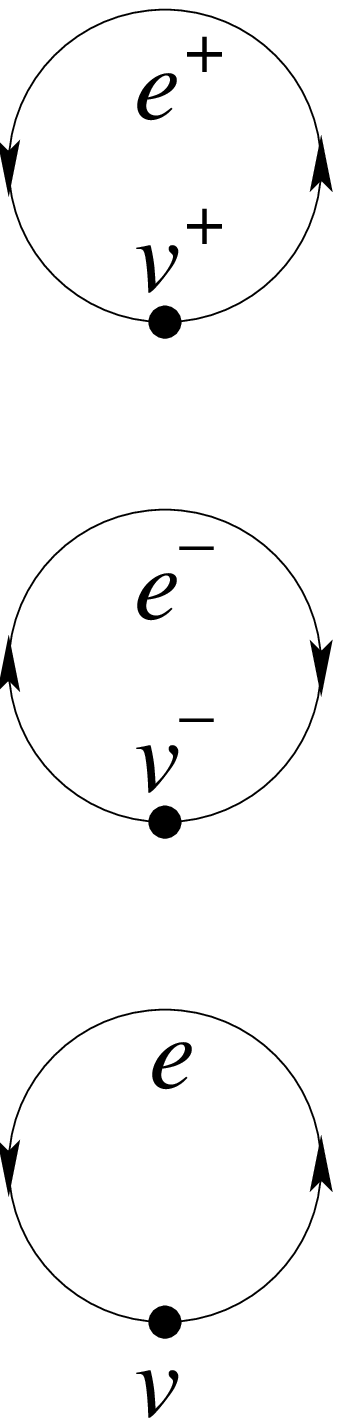}
\hspace{15mm}
\includegraphics[width=18mm]{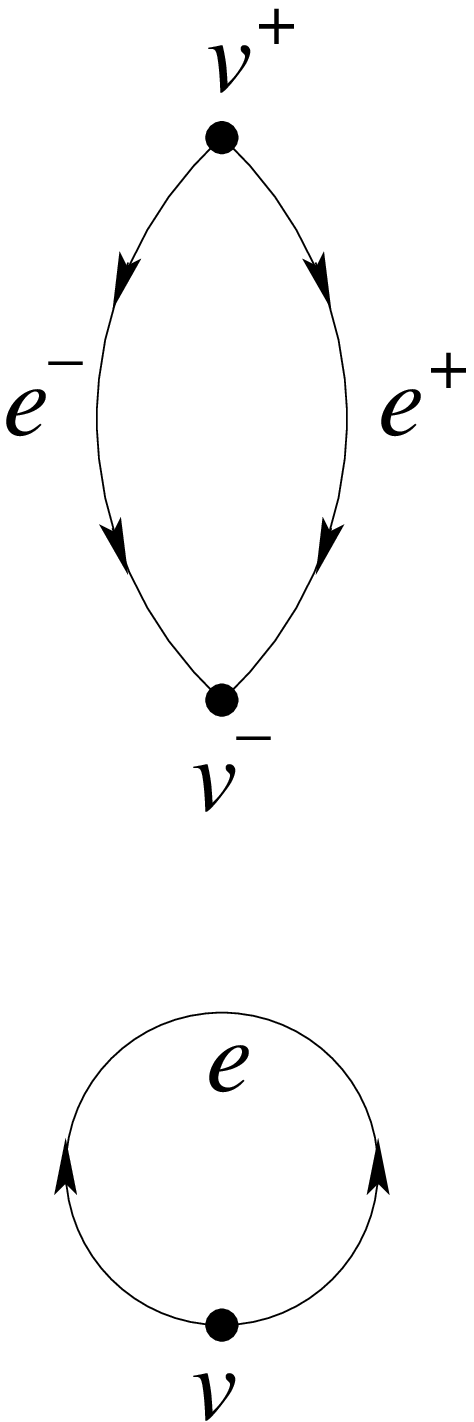}
\hspace{15mm}
\includegraphics[width=18mm]{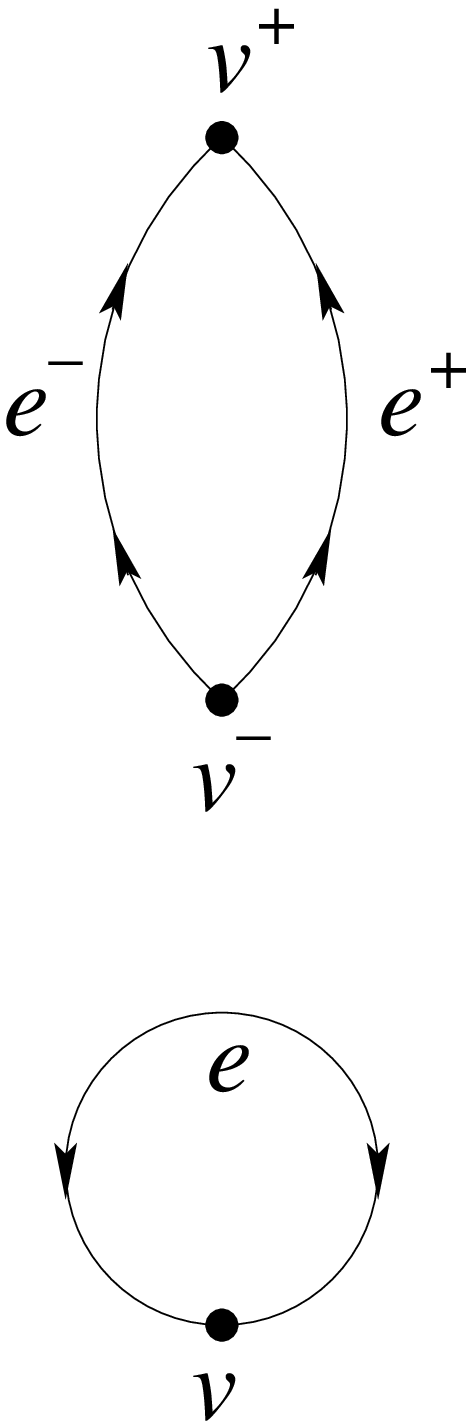}
\caption{Lifting of a loop and its orientation.}\label{Lift-Orientation-Loop}
\label{Fig1}
\end{figure}

The projection $\pi$ maps an edge $e^\beta$ with orientation
$\widetilde\omega(v^\alpha,e^\beta)$ in $\widetilde{\Sigma}$ to the
edge $e$ with orientation
$\omega(v,e):=\alpha\widetilde\omega(v^\alpha,e^\beta)$.

Lifting of orientations preserves coherence. If
$e_1$ and $e_2$ are edges with a common endpoint $v$, with the
orientations $\omega(v,e_1)$ and $\omega(v,e_2)$, then the
concatenated lift $\tilde{e}_1 v^\alpha\tilde{e}_2$ of
$e_1ve_2$, with the orientations
$\widetilde\omega(v^\alpha,\tilde{e}_1)$ and
$\widetilde\omega(v^\alpha,\tilde{e}_2)$, is coherent at
$v^\alpha$ if and only if $e_1ve_2$ is coherent at $v$. This fact
follows from
\[
\widetilde\omega(v^\alpha,\tilde{e}_1)\widetilde\omega(v^\alpha,\tilde{e}_2)
= \alpha\,\omega(v,e_1)\cdot\alpha\,\omega(v,e_2)
= \omega(v,e_1)\,\omega(v,e_2).
\]

Let $\omega_i$ be orientations on subgraphs $\Sigma_i$ of $\Sigma$, $i=1,2$.
The lifted graphs $\widetilde{\Sigma}_i$ are subgraphs of $\widetilde\Sigma$, and the
$(\widetilde{\Sigma}_i,\widetilde{\omega}_i)$ are subgraphs
of the oriented graph $(\widetilde\Sigma,\widetilde\omega)$.
Moreover, lifting of orientations preserves the coupling, that is,
\begin{equation}\label{Lift-Coupling-Preserving}
[\widetilde{\omega}_1,\widetilde{\omega}_2](e^\beta)
=[\omega_1,\omega_2](e)
\end{equation}
for each lift $e^\beta$ of an edge $e$ in $E(\Sigma)$.  Indeed,
\begin{align*}
[\widetilde{\omega}_1,\widetilde{\omega}_2](e^\beta)
&=\widetilde{\omega}_1(v^\alpha,e^\beta)\, \widetilde{\omega}_2(v^\alpha,e^\beta)
=\alpha\,\omega_1(v,e)\cdot\alpha\,\omega_2(v,e)
=[\omega_1,\omega_2](e).
\end{align*}

Let $W$ be a walk in $\Sigma$ of length $n$ with the vertex-edge
sequence $v_0e_1v_1e_2 \cdots v_{n-1}e_nv_n$ and a direction
$\omega_W$. We may lift $W$ to a walk $\tW$ in
$\widetilde{\Sigma}$ as follows:
Select an initial vertex $v^{\alpha_0}$; define
\begin{equation}\label{Lift-of-Walk}
\tW= v^{\alpha_0}_0 \tilde{e}_1 v^{\alpha_1}_1
\tilde{e}_2 \cdots v^{\alpha_{n-1}}_{n-1} \tilde{e}_n v^{\alpha_n}_n, \quad \text{ where }
\alpha_{i}=\alpha_{i-1}\sigma(e_{i}), \quad
\tilde{e}_i=v_{i-1}^{\alpha_{i-1}}v_i^{\alpha_i}.
\end{equation}
We call $\tW$ a \emph{lift} of $W$.
A lift is a {\em resolution} of $W$ if $\tW$ is an open or closed path in $\widetilde{\Sigma}$.
There are exactly two lifts of $W$ since there are exactly two choices for $\alpha_0$.
Moreover, the orientation $\widetilde{\omega}_W$ on the edges in $\tW$ lifted from $\omega_W$ by
\eqref{Lift-Orientation} forms a direction of $\tW$. Thus $(W,\omega_W)$ lifts
to exactly two directed walks, $(\tW, \widetilde{\omega}_W)$ and $(\tW^*, \widetilde{\omega}_W^*)$, where
\begin{gather*}
\tW^* = v^{-\alpha_0}_0 \tilde{e}_1^* v^{-\alpha_1}_1 \tilde{e}_2^* \cdots v^{-\alpha_{n-1}}_{n-1} \tilde{e}_n^* v^{-\alpha_n}_n, \\
\widetilde{\omega}_W^*(v_{i}^{-\alpha_{i}},e_i) =
-\alpha_{i}\omega_W(v_{i},e_i),\quad 1\leq i\leq n.
\end{gather*}
We call $W$ the {\em projection} of $\tW$ and $\tW^*$ and we write $W=\pi(\tW)=\pi(\tW^*)$.

\begin{lem}\label{L:closedwalklift}
\begin{enumerate}[\rm (a)]
\item
The projection $\pi: \widetilde\Sigma \to \Sigma$ induces an
incidence-preserving bijection from $\End(v^\alpha)$ to $\End(v)$
for each $v\in V(\Sigma)$ and $\alpha \in \{+,-\}$.
\item
The projections of directed walks in $\widetilde\Sigma$ are directed walks in $\Sigma$.
The projections of closed walks are closed.
\item
Let $W$ be a closed walk in $\Sigma$ with initial and terminal vertices $v_0$ and $v_n$, and let $\omega_W$ be a direction of $W$. Let $(\tW,\widetilde{\omega}_W)$ be a
lift of the directed walk $(W,\omega_W)$. If $W$ is positive, then
$(\tW,\widetilde{\omega}_W)$ is a directed closed walk. If
$W$ is negative, then $(\tW,\widetilde{\omega}_W)$ is a directed open walk.
\end{enumerate}
\end{lem}
\begin{proof} Parts (a) and (b) are by definition of the projection.

For part (c), let $W=v_0e_1v_1e_2\cdots v_{n-1}e_nv_n$ be a walk in
$\Sigma$, lifted to a walk
$\tW =v_0^{\alpha_0}\tilde{e}_1 v_1^{\alpha_1} \cdots \tilde{e}_nv_n^{\alpha_n}$
in $\widetilde\Sigma$. Since $\sigma(e_i)=\alpha_{i-1}\alpha_i$ by \eqref{Lift-of-Walk},
\[
\sigma(W)=\prod_{i=1}^n \sigma(e_i)
= \prod_{i=1}^n \alpha_{i-1}\alpha_i = \alpha_0 \alpha_n.
\]
Thus, $\sigma(W)=+1$ if and only if $\alpha_n=\alpha_0$. It follows that the walk $\tW$ is closed if and only if $W$ has positive sign.
\end{proof}

\section{Flows}

\subsection{Flows on Signed Graphs}\

The {\em incidence matrix} \cite{MTS} of an oriented signed graph
$(\Sigma,\omega)$ is the $V\times E$ matrix ${\bm M}={\bm M}(\Sigma,\omega)=[{\bm m}(v,e)]$,
where $\bm m$ is a function ${\bm m}: V\times E\to{\mathbb Z}$ defined by
\begin{equation}\label{incidence-matrix}
{\bm m}(v,e)=\sum_{v\in\End(e)}\omega(v,e)
=\begin{cases}
\omega(v,e) & \text{if $e$ is a link},\\
2\,\omega(v,e) & \text{if $e$ is a negative loop},\\
0 & \text{otherwise}.
\end{cases}
\end{equation}
When $e$ is a loop with $\End(e)=\{v,v\}$ we add $\omega(v,e)$ for each end of $e$, so the sum is
$0$ when $e$ is a positive loop and is $2\,\omega(v,e)$ when $e$ is a negative loop.

An \emph{integral flow} on an oriented signed graph $(\Sigma,\omega)$ is a function $f: E(\Sigma) \to \mathbb Z$ which is
\emph{conservative} at every vertex, meaning that the net contribution to each vertex is zero.
The {\em boundary operator} of $(\Sigma,\omega)$, $\partial:{\mathbb Z}^{E(\Sigma)}\to{\mathbb Z}^{V(\Sigma)}$, is defined by
\begin{equation}\label{Boundary-Operator}
\partial f(v) = \sum_{e \in E} {\bm m}(v,e) f(e)
= \sum_{(v,e)\in\End(v)} \omega(v,e)f(e).
\end{equation}
Thus a function $f: E(\Sigma) \to \mathbb Z$ is a flow if and only if $\partial f$ is identically zero.

The set of all integral flows on $(\Sigma,\omega)$ forms a $\mathbb Z$-module, called the
\emph{flow lattice} by Chen and Wang, who developed its basic theory
in \cite{Chen-Wang1}. One can define flows with values in an
arbitrary abelian group, for example, the additive group of real numbers and finitely
generated abelian groups.
Many of the following remarks are applicable for such flows.
We omit the word ``integral'' when mentioning integral flows.

The theory of flows on signed graphs depends essentially on the
graph and the sign function but not on the orientation, since only the notation changes when edges are reoriented.
Specifically, a flow $f$ on $\Sigma$ with respect to an orientation $\omega$ represents the same flow on $\Sigma$ as $[\omega,\rho]f$ on $\Sigma$ with respect to another orientation $\rho$.  (Therefore it is correct to speak of ``flows on signed graphs''.)

The \emph{support} of a function $f: E(\Sigma) \to \mathbb Z$ is the set of
edges $e$ such that $f(e)\neq 0$; it is denoted by $\supp f$. We denote by
$\Sigma(f)$ the subgraph of $\Sigma$ whose edge set is $\supp f$ and whose vertex set consists of vertices incident with edges in $\supp f$. The \emph{zero flow} is the flow that is zero on all edges.
Flows other than the zero flow are referred to as {\em nonzero flows}.
A \emph{circuit flow} of a signed graph (as defined in \cite{Chen-Wang1}) is a flow
whose support is a signed-graph circuit, having values $\pm1$ on the edges of the
circles and $\pm2$ on the edges of the circuit path (for Type
III circuits). See Figure~\ref{Circuit-II-III} for circuit flows of
Type II and Type III.
\begin{figure}[h]
\centering \subfigure[Type
II]{\includegraphics[height=18mm]{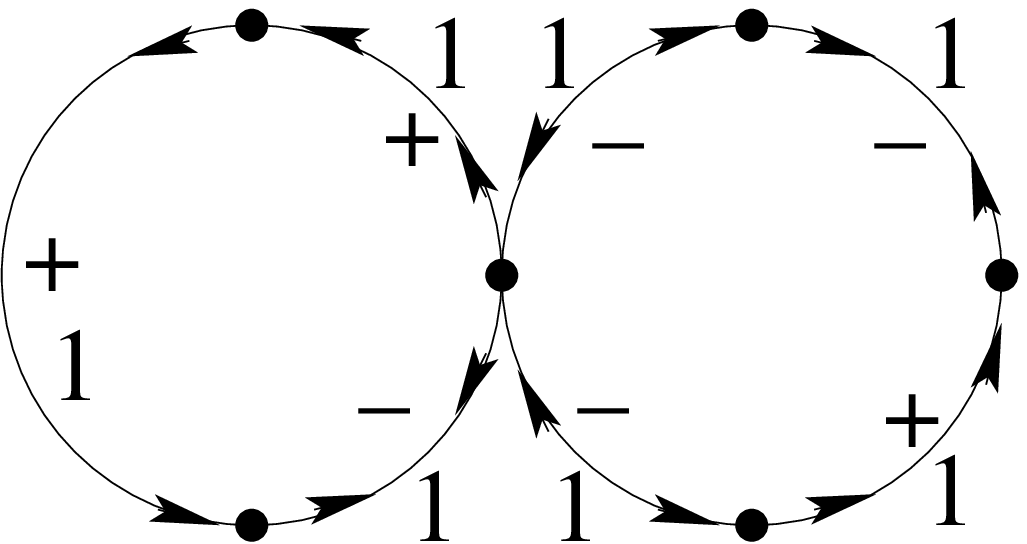}} \hspace{15mm}
\subfigure[Type III]{\includegraphics[height=18mm]{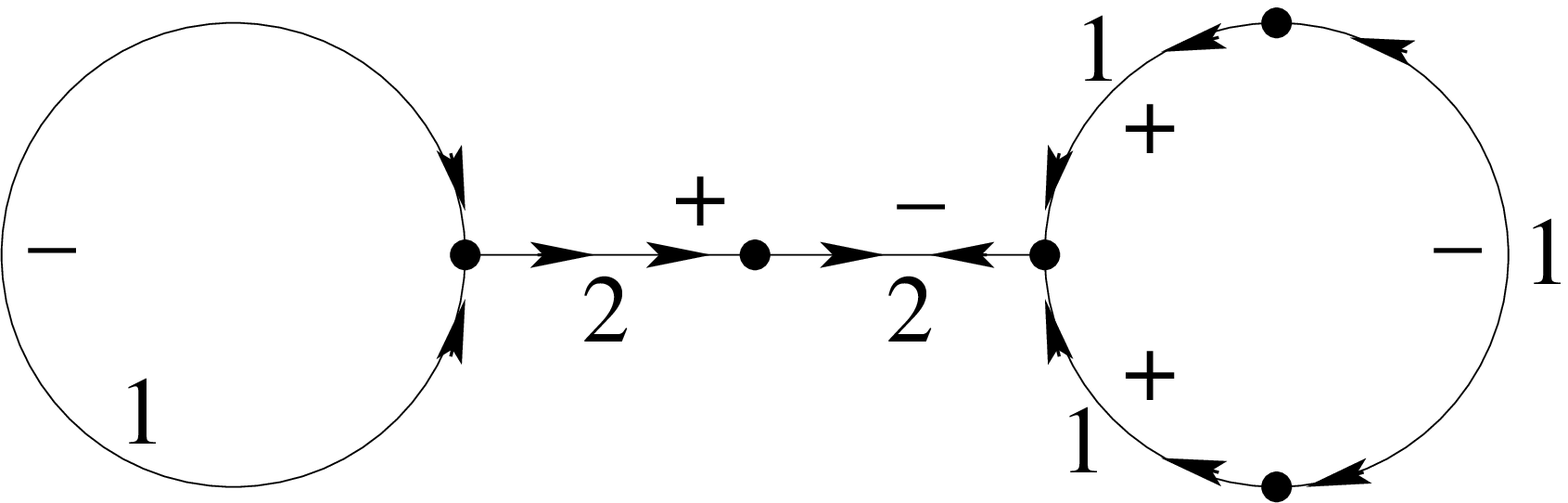}}
\caption{Signed-graph circuit flows of Types II and III.}\label{Circuit-II-III}
\label{Fig-circuits}
\end{figure}

An integral flow $f_1$ \emph{conforms to the sign
pattern of} $f$ if $\supp f_1 \subseteq \supp f$ and $f_1(e)$ has
the same sign as $f(e)$ for all edges $e$ in $\supp f_1$.

An integral flow $f$ on $(\Sigma,\omega)$ lifts to a flow on the
oriented double covering graph
$(\widetilde\Sigma,\widetilde\omega)$, possibly in more than one
way. The best way to see existence of a lift is through the
correspondence between integral flows and walks (when $\Sigma(f)$ is
connected).

A directed closed, positive walk $(W,\omega_W)$ on
$(\Sigma,\omega)$ corresponds to a unique integral flow
$f_{(W,\,\omega_W)}$, defined by
\begin{equation}\label{Flow-Walk-Defn}
f_{(W,\,\omega_W)}(e)=\sum_{e_i\in W,\; e_i=e} [\omega,\omega_W](e_i),
\end{equation}
where $W$ is viewed as a multiset $\{e_1,\ldots,e_n\}$ of edges if
$W=v_0e_1v_1\cdots e_nv_n$ (see \cite{Chen-Wang1}). Clearly,
$f_{(W,\,\omega_W)}(e)$ is the number of times $W$ traverses
the edge $e$ with $\omega_W$ and $\omega$ agreeing, minus the number
of times $W$ traverses $e$ while $\omega_W$ disagrees with $\omega$.
In case the direction $\omega_W$ is the same as $\omega$ restricted to $W$, we simply write $f_{(W,\,\omega_W)}$ as $f_W$ and we have
\begin{equation*}
f_W(e) = \big|\{e_i\in W:e_i=e\}\big| \;\; \text{(as a multiset)}.
\end{equation*}
To see why $f_{(W,\,\omega_W)}$ is a flow, consider the
contribution to $f_{(W,\,\omega_W)}$ of a pair of consecutive
edges, $e_iv_ie_{i+1}$, at the intervening vertex $v_i$. Since
$(W,\omega_W)$ is coherent at $v_i$, the contribution of these
edges to $\partial f_{(W,\,\omega_W)}(v_i)$ is $0$.  This same
argument applies to the initial vertex if we take subscripts modulo
the length of $W$.

We can apply the same definition of $f_{(W,\,\omega_W)}$ to any directed walk $(W,\omega_W)$, not necessarily closed or positive, but then the result may no longer be a flow. In
fact, we have the following lemma.

\begin{lem}\label{L:walkflow}
Let $(W,\omega_W)$ be a directed walk with $W = v_0e_1v_1e_2\cdots v_{n-1}e_nv_n$, $n>0$. Then the function
$f_{(W,\,\omega_W)}$ is conservative everywhere except possibly at $v_0$ and $v_n$; that is,
$
\partial f_{(W,\,\omega_W)}(v) = 0$ for all $v\neq v_0,v_n.
$
If $W$ is closed, then
\[
\partial f_{(W,\,\omega_W)}(v_0) = \begin{cases}
0    & \text{when $W$ is positive,} \\
2\omega_W(v_0,e_1) & \text{when $W$ is negative.}
\end{cases}
\]
If $W$ is open, then
\[
\begin{aligned}
\partial f_{(W,\,\omega_W)}(v_0) &= \omega_W(v_0,e_1), \\
\partial f_{(W,\,\omega_W)}(v_n) &= \omega_W(v_n,e_n) =-\sigma(W)\,\omega_W(v_0,e_1).
\end{aligned}
\]
\end{lem}
\begin{proof} Incoherence of $(W,\omega_W)$ can only possibly occur at $v_0$ and $v_n$.
Fix a vertex $v$. We have
\begin{align*}
\partial f_{(W,\,\omega_W)}(v)
&= \sum_{e\in E}{\bm m}(v,e) \sum_{e_i\in W,\, e_i=e} [\omega,\omega_W](e_i) \\
&= \sum_{e_i\in W} {\bm m}(v,e_i) [\omega,\omega_W](e_i) \\
&= \sum_{e_i\in W,\,v_{i-1}=v} \omega(v_{i-1},e_i) [\omega,\omega_W](e_i) + \sum_{e_i\in W,\,v_{i}=v} \omega(v_i,e_i) [\omega,\omega_W](e_i) \\
&= \sum_{e_i\in W,\,v_{i-1}=v} \omega_W(v_{i-1},e_i) + \sum_{e_i\in W,\,v_{i}=v} \omega_W(v_i,e_i) \\
&= \sum_{e_i\in W} \bm{m}(v,e_i).
\end{align*}
The second to last equality follows from the definition of coupling.

Let $v$ appear in the vertex-edge sequence of $W$ as
$v_{\ell_1},v_{\ell_2},\ldots,v_{\ell_m}$.
If $l_i\neq0, n$,
then $\omega_W(v_{l_i},e_{\ell_i})+\omega_W(v_{l_i},e_{\ell_i+1})=0$.
Thus
\[
\partial f_{(W,\,\omega_W)}(v) = \begin{cases}
\omega_W(v_0,e_1)+\omega_W(v_n,e_n) &\text{if } v_0=v_n=v, \\
\omega_W(v_0,e_1)  &\text{if } v_0=v \neq v_n, \\
\omega_W(v_n,e_n)  &\text{if } v_n=v \neq v_0, \\
0  &\text{if } v_0,v_n \neq v,
\end{cases}
\]
regardless of whether $W$ closed or open.

In the case that $W$ is closed and $v=v_0=v_n$, then
\[
\partial f_{(W,\,\omega_W)}(v) = \omega_W(v,e_1) [1 - \sigma(W)]
\]
by Lemma \ref{L:walksign}.
If $\sigma(W)$ is positive, then $\partial f_{(W,\,\omega_W)}(v_0)=0$.
If $\sigma(W)$ is negative, then $\partial f_{(W,\,\omega_W)}(v_0) = 2\omega_W(v_0,e_1) = 2\omega_W(v_n,e_n).$

In the case that $W$ is an open walk, namely, $v_0\neq v_n$, we have
$\partial f_{(W,\,\omega_W)}(v_0)=\omega_W(v_0,e_1)$ and
$\partial f_{(W,\,\omega_W)}(v_n)=\omega_W(v_n,e_n)=-\sigma(W)\omega_W(v_0,e_1)$.
\end{proof}

Conversely, directed closed walks can be constructed (though usually not uniquely) from integral flows.

\begin{prop}\label{P:flowwalk}
If $f$ is a nonnegative, nonzero integral flow on
$(\Sigma,\omega)$ such that $\Sigma(f)$ is connected, then there
exists a directed closed, positive walk $(W,\omega)$ on
$\Sigma(f)$ such that $f_W=f$.
\end{prop}
\begin{proof}
We apply induction on the total weight of $f$,
\[
\|f\| := \sum_{e\in E(\Sigma)} |f(e)|.
\]
To begin, choose a vertex $v_0$ and an edge $e_1$ incident with $v_0$ in $\Sigma(f)$.
Let $v_1$ be the other endpoint of $e_1$ ($v_1=v_0$ if $e_1$ is a loop).
This gives a walk $W_1=v_0e_1v_1$ of length $1$.  Clearly, $f\geq f_{W_1}\geq 0$.

Assume that we have constructed a partial walk
$W_k=v_0e_1v_1\cdots e_kv_k$ on $\Sigma(f)$ with $k\geq 1$, and that $f\geq f_{W_k}\geq 0$.
If $W_k$ is not closed and positive, then by Lemma~\ref{L:walkflow} the function $f_{W_k}$ is not a flow, for it
is not conservative at $v_k$; indeed,
$\partial f_{W_k}(v_k)=\omega(v_k,e_k)$ (or $2\omega(v_k,e_k)$ if $v_0=v_k$).
Since $f$ is conservative at $v_k$, $\partial(f-f_{W_k})=-\omega(v_k,e_k)$ or $-2\omega(v_k,e_k)$.
Since $f\geq0$, there exists an edge $e_{k+1}$ incident with $v_k$ in
$\Sigma(f-f_{W_k})$ such that
$\omega(v_k,x_{k+1})=-\omega(v_k,e_k)$.
Let $v_{k+1}$ denote the other endpoint of $e_{k+1}$ and extend $W_k$ to a walk
$W_{k+1}:=W_ke_{k+1}v_{k+1}$ on $\Sigma(f)$.
We have $f\geq f_{W_{k+1}}\geq 0$ by the construction.
Continuing this procedure as long as $W_{k+1}$ is not a closed, positive
walk, we finally obtain a directed closed, positive walk $(W_n,\omega)$.
Then $f':=f-f_{W_n}$ is a nonnegative integral flow and $\|f'\|<\|f\|$.

Let $\Sigma(f')$ have connected components
$\Sigma_1,\ldots,\Sigma_m$. Set $f'_i=f'|_{\Sigma_i}$,
$i=1,\ldots,m$. Note that $f'=\sum_{i=1}^m f'_i$ and $\supp
f'_i=E(\Sigma_i)\subseteq\Sigma(f)$, $1\leq i\leq m$. Each $f'_i$ is
a nonnegative, nonzero integral flow on $(\Sigma,\omega)$ and
satisfies $\|f'_i\|<\|f\|$. By induction there exists a directed
closed, positive walk $(W'_i,\omega)$ on $\Sigma_i$ such that
$f'_i=f_{W'_i}$. The union of $(W'_1,\omega),\ldots,(W'_m,\omega)$
and $(W_n,\omega)$ is connected. One can construct a single directed
closed, positive walk $(W,\omega)$ by rearranging the initial and
terminal vertices of $W'_1,\ldots,W'_m$ and $W_n$, and connecting
them properly at some of their intersections. Then $W$ is a walk on
$\Sigma(f)$ and $f_{W}=f$.
\end{proof}

Let $f$ be an integral flow on $(\Sigma,\omega)$. Associated
with $f$ is an orientation $\omega_f$ on $\Sigma$ defined by
\begin{equation}\label{orientation-f}
\omega_f(v,e)=\begin{cases}
\phantom{-}\omega(v,e) &    \text{if } f(e)\geq 0,\\
-\omega(v,e) &  \text{if } f(e)<0,
\end{cases}
\end{equation}
for each edge $e$ and endpoint $v \in \End(e)$.
The {\em absolute function} $|f|$ is defined by
$$|f|(e)=|f(e)| \quad \text{for } e\in E(\Sigma).$$
Since $|f|=[\omega,\omega_f]f$, $|f|$ is a
nonnegative, nonzero integral flow on $(\Sigma,\omega_f)$.

\begin{cor}\label{flow-walk}
Let $f$ be a nonzero integral flow on $(\Sigma,\omega)$.
\begin{enumerate}[\rm (a)]
\item
If $\Sigma(f)$ is connected, then there exists a directed
closed, positive walk $(W,\omega_f)$ on $\Sigma(f)$ such that
$f=f_{(W,\,\omega_f)}$.
\item
There exists a directed closed, positive walk $(W,\omega_W)$
such that $f=f_{(W,\,\omega_W)}$.
\end{enumerate}
\end{cor}
\begin{proof}
(a) Note that $|f|$ is a nonnegative, nonzero integral flow on
$(\Sigma,\omega_f)$ and $\Sigma(|f|)=\Sigma(f)$. According to
Proposition~\ref{P:flowwalk}, there exists a directed closed,
positive walk $(W,\omega_f)$ on $\Sigma(|f|)$ such that
$f_{W}=|f|$ within the oriented signed graph $(\Sigma,\omega_f)$, where
\begin{equation*}
f_{W}(e) = \sum_{e_i\in W,\,e_i=e}[\omega_f,\omega_f](e_i) = \sum_{e_i\in W,\,e_i=e}1
\end{equation*}
for each $e\in E(\Sigma).$
For the same directed closed, positive walk $(W,\omega_f)$ within $(\Sigma,\omega)$, we have
\begin{align*}
f_{(W,\,\omega_f)}(e)
&= \sum_{e_i\in W,\,e_i=e}[\omega,\omega_f](e_i)
= [\omega,\omega_f](e) \sum_{e_i\in W,\,e_i=e}1 \\
&=  [\omega,\omega_f](e)f_{W}(e)
\end{align*}
for each $e \in E(\Sigma)$.
Since $f=[\omega,\omega_f]\, |f|$ and $|f|=f_{W}$, it
follows that $f_{(W,\,\omega_f)}=f$.

(b) Let $\Sigma(f)$ have components $\Sigma_i$, $1\leq i\leq m$,
where $m>1$. Since each $\Sigma(f_i)$ is connected, there are
directed closed, positive walks $(W_i,\,\omega_{f_i})$ on $\Sigma_i$
such that $f_i=f_{(W_i,\,\omega_{f_i})}$. Let $v_i$ be the initial
and terminal vertex of $W_i$. Let $(P_i,\omega_{P_i})$ be a directed
path from $v_1$ to $v_i$, where $i=2,\ldots,m$. Then
\[
(W,\omega_W):=(W_1P_2W_2P_2^{-1}\cdots P_mW_mP_m^{-1},
\omega_{W_1}\omega_{P_2}\omega_{W_2}\omega_{P_2}^{-1}\cdots
\omega_{P_m}\omega_{W_m}\omega_{P_m}^{-1})
\]
is a directed closed, positive walk such that $f=f_{(W,\,\omega_W)}$.
\end{proof}

\subsection{Lifted Flows}\

Consider a function $\tilde{f}: E(\widetilde\Sigma)\to \mathbb Z$
defined on the edge set of the double covering graph
$\widetilde\Sigma$.  The \emph{projection} of $\tilde f$ is the
function $\pi(\tf): E(\Sigma) \to \mathbb Z$ defined by
\[
\pi(\tf)(e) = \tf(\tilde e) + \tf(\tilde e^*), \quad e\in E(\Sigma).
\]
Let $\widetilde{\bm M}={\bm M}(\widetilde\Sigma,\widetilde\omega)=[{\bm m}(v^\alpha,\tilde e)]$ denote
the $V(\widetilde\Sigma)\times E(\widetilde\Sigma)$ incidence matrix of $(\widetilde\Sigma,\widetilde\omega)$;
it is defined as in \eqref{incidence-matrix} but now every edge is positive.
Since only positive loops of $\Sigma$ are lifted to loops in $\widetilde\Sigma$, we have
\[
\label{Lift-Matrix-Entry}
{\bm m}(v^\alpha,\tilde e) =
\sum_{v^\alpha\in\End(\tilde{e})} \tilde\omega(v^\alpha,\tilde e) =
\begin{cases}
0 & \text{if $e$ is a positive loop,} \\
\alpha\,\omega(v,e) & \text{otherwise,}
\end{cases}
\]
where $e=\pi(\tilde e)$, as usual (and the sum is over two edge ends if $\tilde e$ is a loop at $v^\alpha$).
Also as usual, the boundary operator
$\partial:{\mathbb Z}^{E(\widetilde\Sigma)}\to{\mathbb Z}^{V(\widetilde\Sigma)}$
is defined by
\[
\partial\tilde f(v^\alpha)=\sum_{\tilde e\in E(\widetilde\Sigma)}{\bm m}(v^\alpha,\tilde e)\tilde f(\tilde e)
=\sum_{(v^\alpha,\tilde e)\in\End(v^\alpha)} \tilde\omega(v^\alpha,\tilde e) f(\tilde e).
\]

\begin{lem}\label{L:Lift-Boundary-Operator}
Let $\tilde f$ be a function defined on $E(\widetilde\Sigma)$. Then
$\partial\pi(\tilde f)$ is given by
\begin{equation}\label{Lift-Bounary-Operator}
\partial\pi(\tilde f)(v)=\partial\tilde f(v^+)-\partial\tilde f(v^-), \quad v\in V.
\end{equation}
If $\tilde f$ is a flow on $(\widetilde\Sigma,\widetilde\omega)$, so is
$\pi(\tilde f)$ on $(\Sigma,\omega)$.
\end{lem}
\begin{proof}
Fix a vertex $v$ in $V(\Sigma)$. Note that $\pi$ acts as a bijection
between $\End(v^\alpha)$ in $\tilde\Sigma$ and $\End(v)$ in
$\Sigma$, for $\alpha\in\{+1,-1\}$, and recall that
$\tilde\omega(v^\alpha,\tilde{e})=\alpha\omega(v,e)$. Then
\begin{align*}
\partial\pi(\tilde f)(v)&=\sum_{(v,e)\in\End(v)} \omega(v,e) \pi(\tilde f)(e) \\
&=\sum_{(v,e)\in\End(v)} \omega(v,e) [\tilde f(\tilde e)+\tilde f(\tilde e^*)] \\
&=\sum_{(v,e)\in\End(v)} \omega(v,e) \tilde f(\tilde e) + \sum_{(v,e)\in\End(v)} \omega(v,e) \tilde f(\tilde e^*) \\
&=\sum_{(v,e)\in\End(v)} \tilde\omega(v^+,\tilde e) \tilde f(\tilde e) - \sum_{(v,e)\in\End(v)} \tilde\omega(v^-,\tilde e^*) \tilde f(\tilde e^*) \\
&=\sum_{(v^+,\tilde e)\in\End(v^+)} \tilde\omega(v^+,\tilde e) \tilde f(\tilde e) - \sum_{(v^-,\tilde e^*)\in\End(v^-)} \tilde\omega(v^-,\tilde e^*) \tilde f(\tilde e^*) \\
&=\partial\tilde f(v^+) - \partial\tilde f(v^-).
\end{align*}

When $\tilde f$ is a flow on $(\widetilde\Sigma,\widetilde\omega)$,
then $\partial\tilde f(v^+)=\partial\tilde f(v^-)=0$. Thus
$\partial\pi(\tilde f)(v)=0$ by \eqref{Lift-Bounary-Operator}, so $\partial\pi(\tilde f)$ is a flow on $(\Sigma,\omega)$.
\end{proof}

A \emph{lift} of an integral flow $f$ of $(\Sigma,\omega)$ to
$\widetilde\Sigma$ is an integral flow $\tf$ of $(\widetilde\Sigma,\widetilde\omega)$
such that $\pi(\tilde f) = f$.

\begin{prop}\label{P:flowlift}
\begin{enumerate}[\rm (a)]
\item
Let $(\tW,\widetilde{\omega}_W)$ be a lift of a
directed walk $(W,\omega_W)$. If $W$ is closed and has positive
sign, then $(\tW,\widetilde{\omega}_W)$ is a directed closed walk, and
\begin{equation}\label{Lift-Walk-Flow}
\pi\big(f_{(\tW,\,\widetilde{\omega}_W)}\big)=f_{(W,\,\omega_W)}.
\end{equation}
\item
Let $f$ be an integral flow on $(\Sigma,\omega)$ such that
$f=f_{(W,\,\omega_f)}$, where $(W,\omega_f)$ is a directed closed
positive walk. If $(W,\omega_f)$ is lifted to a directed closed walk
$(\tW,\widetilde{\omega}_f)$ in $\widetilde\Sigma$, then
$f_{(\tW,\,\widetilde{\omega}_f)}$ is a flow on
$(\widetilde\Sigma,\widetilde\omega)$ lifted from $f$. Moreover, if
$f\geq 0$, then $f_{(\tW,\,\widetilde{\omega}_f)}\geq 0$.
\end{enumerate}
\end{prop}
\begin{proof}
(a) It follows from Lemma~\ref{L:closedwalklift} that $(\tW,\widetilde{\omega}_W)$ is a directed closed walk.
To see that $\pi(f_{(\tW,\,\widetilde{\omega}_W)})=f_{(W,\,\omega_W)}$, it suffices to show that
$f_{(W,\,\omega_W)}(e)=f_{(\tW,\,\widetilde{\omega}_W)}(e^+)+f_{(\tW,\,\widetilde{\omega}_W)}(e^-)$.
Since lifting of orientations preserves the coupling, by \eqref{Lift-Coupling-Preserving} we have
\begin{eqnarray*}
\pi\big(f_{(\tW,\,\widetilde{\omega}_W)}\big)(e)
&=& f_{(\tW,\,\widetilde{\omega}_W)}(e^+)+f_{(\tW,\,\widetilde{\omega}_W)}(e^-) \\
&=& \sum_{\tilde{e}_i\in\tW,\; \tilde{e}_i\in\{e^+,e^-\}} [\widetilde\omega,\widetilde{\omega}_W](\tilde{e}_i) \\
&=& \sum_{e_i\in W,\; e_i=e} [\omega,\omega_W](e_i) \\
&=& f_{(W,\,\omega_W)}(e).
\end{eqnarray*}

(b) Since $\omega_f$ is an orientation on $\Sigma$,
$\widetilde{\omega}_f$ is an orientation on $\widetilde\Sigma$. Viewing $\omega_f$ as a
direction $\omega_W$ on $W$, $\widetilde{\omega}_W$ is a
direction of $\tW$. Note that the lift of $\omega_W$ is
the same as the lift of $\omega_f$, so we can view
$\widetilde{\omega}_f$ as the direction $\widetilde{\omega}_W$ of $\tW$.

Since $f=f_{(W,\,\omega_f)}$, it follows from
\eqref{Lift-Walk-Flow} and Lemma \ref{L:Lift-Boundary-Operator} that $f_{(\tW,\,\widetilde{\omega}_f)}$ is a flow on
$(\widetilde\Sigma,\widetilde\omega)$ and a lift of $f$.
If $f\geq0$, then $\omega_f=\omega$; consequently,
$\widetilde{\omega}_f=\widetilde\omega$.
Thus $f_{(\tW,\,\widetilde{\omega}_f)}\geq 0$ by the definition in~\eqref{Flow-Walk-Defn}.
\end{proof}

\subsection{Decomposability}\

An integral flow $f$ is \emph{conformally decomposable}
if it is nonzero and can be represented as a sum of two other
integral flows, $f = f_1 + f_2$, each of which is nonzero and
conforms to the sign pattern of $f$, that is, $f_1(e)f_2(e)\geq 0$
for all edges $e$; this means that $f_1(e)$ and $f_2(e)$ have
the same sign when they are nonzero. An integral flow is said
to be {\em conformally indecomposable} if it is nonzero and
not conformally decomposable.
It is well known and easy to see that conformally indecomposable flows on an unsigned graph are just graphic circuit flows.

A nonnegative, nonzero integral flow $f$ is
{\em minimal} provided that if $g$ is a nonnegative, nonzero
integral flow on $(\Sigma,\omega)$ such that $g(e)\leq f(e)$
for all edges $e$, then $g=f$.

If an integral flow $f$ is nonnegative, then its minimality is
equivalent to its conformal indecomposability. In fact, if $f$ is
conformally decomposed into $f=f_1+f_2$, then $f_1$ and $f_2$ must
be nonnegative, nonzero integral flows such that $f_1\leq f$ and
$f_1\neq f$; this means that $f$ is not minimal. Conversely, if $f$
is not minimal, then there is a nonnegative, nonzero integral flow
$g$ on $(\Sigma,\omega)$ such that $g\leq f$ but $g\neq f$. Now
$f-g$ is nonzero and nonnegative, and $f$ decomposes conformally
into $g$ and $f-g$.

The following proposition shows that conformal indecomposability
of a nonzero integral flow $f$ on $(\Sigma,\omega)$ is equivalent
to the minimality of the flow $|f|$ on $(\Sigma,\omega_f)$,
where $|f|$ is the absolute value function of $f$ and $\omega_f$ is
the orientation given by \eqref{orientation-f}.

\begin{prop}\label{L:Indecomposability-Minimality}
Let $f$ be a nonzero integral flow on $(\Sigma,\omega)$. Then the
following properties are equivalent.
\begin{enumerate}[\rm (a)]
\item $f$ is a conformally indecomposable flow on $(\Sigma,\omega)$.

\item $|f|$ is a conformally indecomposable flow on
$(\Sigma,\omega_f)$.

\item $|f|$ is a minimal flow on $(\Sigma,\omega_f)$.
\end{enumerate}
\end{prop}
\begin{proof}
Applying the boundary operator \eqref{Boundary-Operator}, it is
clear that $f$ is a flow on $(\Sigma,\omega)$ if and only if
$|f|=[\omega,\omega_f]\, f$ is a flow on $(\Sigma,\omega_f)$. Since
$|f|$ is nonnegative, its minimality is equivalent to its conformal
indecomposability, so we already have $(b)\Leftrightarrow(c)$ by the
argument above.

(a) $\Rightarrow$ (c): Suppose $|f|$ is not minimal, that is,
$|f|=g_1+g_2$, where $g_1$ and $g_2$ are nonnegative, nonzero
integral flows on $(\Sigma,\omega_f)$. Setting
$f_i=[\omega,\omega_f]\, g_i$ yields nonzero integral flows on
$(\Sigma,\omega)$, $i=1,2$. Thus
\[
f=[\omega,\omega_f]\, |f|
=[\omega,\omega_f]\, g_1
+[\omega,\omega_f]\, g_2 =f_1+f_2,
\]
and $f_1\, f_2=g_1\, g_2\geq 0$, meaning that $f$ is conformally
decomposable. This is a contradiction.

(b) $\Rightarrow$ (a): Suppose $f$ is conformally decomposable, that
is, $f=f_1+f_2$, where $f_1$ and $f_2$ are nonzero integral flows
on $(\Sigma,\omega)$ such that $f_1\, f_2\geq 0$. Setting
$g_i=[\omega,\omega_f]\, f_i$ yields nonzero integral flows on
$(\Sigma,\omega_f)$, $i=1,2$. For each edge $e$, if $f_i(e)>0$, we
must have $f(e)>0$ and $[\omega,\omega_f](e)=1$ by the definition of
$\omega_f$; if $f_i(e)<0$, we must have $f(e)<0$ and
$[\omega,\omega_f](e)=-1$; thus $g_i(e)\geq 0$. Hence
\[
|f|=[\omega,\omega_f]\, f
=[\omega,\omega_f]\, f_1+[\omega,\omega_f]\, f_2 =g_1+g_2,
\]
meaning that $|f|$ is conformally decomposable. This is a contradiction.
\end{proof}

\section{Indecomposable Flows}

A signed graph with nonempty edge set is called {\em sesqui-Eulerian} if there
exists a directed closed, positive walk that uses every edge at least once
but at most twice, and whose direction has the same
orientation on each pair of repeated edges. A sesqui-Eulerian signed graph
is {\em prime} if no such directed closed, positive walk
properly contains any directed closed, positive subwalks.
A sesqui-Eulerian signed graph is {\em minimal} if it does not
properly contain any sesqui-Eulerian signed subgraphs.
It is clear that minimal sesqui-Eulerian signed graphs must be prime.

\begin{defn}\label{cycle-tree-defn}
A signed graph $T$ with nonempty edge set is called a {\em circle-tree} if it satisfies the following conditions:
\begin{enumerate}[\hspace{2ex} \rm (a)]
\item
$T$ is connected.
\item
Each block of $T$ is either a circle or an edge.
\item
Each end block of $T$ is a circle.
\item
Each cut-vertex is incident with exactly two blocks.
\end{enumerate}
The blocks that are circles are called the \emph{circle blocks} of $T$.
The paths (of possibly zero length) between pairs of circle blocks are called the {\em block paths} of $T$.
The \emph{length} of $T$ is
\[
\ell(T):=\sum_i \ell(C_i)+2\sum_j\ell(P_j),
\]
where the $C_i$ are the circle blocks and the $P_j$ are the block paths.

A circle-tree is said to be {\em sesqui-Eulerian} if it further satisfies
\begin{enumerate}[\hspace{2ex} \rm (e)]
\item
{\em Parity Condition}: The sign of a circle block equals $(-1)^p$, where $p$ is the number of cut-vertices of $T$ on the circle.
\end{enumerate}
\end{defn}

We shall see that prime sesqui-Eulerian signed graphs are
sesqui-Eulerian circle-trees (Proposition \ref{prime}) and that minimal sesqui-Eulerian signed
graphs are signed-graph circuits, i.e., circuits of Types I, II, and III (Corollary \ref{minimal}).

A circle-tree can be viewed as a tree-like graph whose ``vertices''
are the circle blocks and whose ``edges'' are the block paths. The
endpoints of block paths are cut-vertices. A block path of length zero
is a common cut-vertex of two circle blocks. If each circle block is
contracted to a point, the resulting graph is a tree.

We may also think of a sesqui-Eulerian circle-tree as a ``tree'' whose
``vertices'' are its vertex-disjoint maximal Eulerian subgraphs and whose ``edges'' are the paths (of
positive length) between the maximal Eulerian subgraphs, where each
such maximal Eulerian subgraph is also a tree-like structure whose
``vertices'' are edge-disjoint circles and ``edges'' are the
intersection vertices between pairs of circles.

Let $T$ be a sesqui-Eulerian circle-tree. A {\em direction} of $T$ is an
orientation $\omega_T$ on the signed graph $T$ such that
$(T,\omega_T)$ has neither a sink nor a source, and
for each circle block $C$ the subgraph $(C,\omega_T)$ has either a sink or a
source at each cut-vertex of $T$ on $C$. It is easy to see that
there exist exactly two (opposite) directions on $T$.
Figure~\ref{Prime-Sesqui-Eulerian-Signed-Graph-Exmp} exhibits a sesqui-Eulerian circle-tree with a direction.
\begin{figure}[h]
\includegraphics[width=100mm]{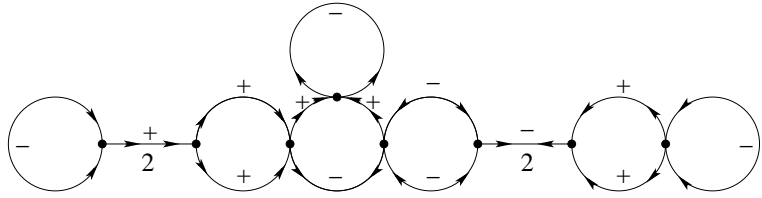}
\caption{A sesqui-Eulerian circle-tree with a direction.}
\label{Prime-Sesqui-Eulerian-Signed-Graph-Exmp}
\end{figure}

For unsigned graphs, since all edges are positive, sesqui-Eulerian circle-trees are just
circles, and directed sesqui-Eulerian circle-trees are directed circles.

When a circle-tree $T$ is contained in a signed graph $\Sigma$, the {\em indicator function} of $T$ is the function $I_T: E(\Sigma)\to{\mathbb Z}$ defined by
\[
I_T(e) = \begin{cases}
1 & \text{if $e$ belongs to a circle block},\\
2 & \text{if $e$ belongs to a block path},\\
0 & \text{otherwise}.
\end{cases}
\]
The {\em characteristic vector} of $(T,\omega_T)$
within the oriented signed graph $(\Sigma,\omega)$ is the function $[\omega,\omega_T]\,I_T$.

A {\em minimal tour} on $T$ is a closed walk that uses every edge of $T$ and has minimum length. A subgraph of $T$ is a {\em circle-subtree} if it is a circle-tree and its circle blocks and block paths are circle blocks and block paths of $T$.

\begin{prop}[Existence and Uniqueness of Direction on a Sesqui-Eulerian Circle-Tree]
\label{Prime-Sesqui-Eulerian-Property}
If $T$ is a sesqui-Eulerian circle-tree, then:
\begin{enumerate}[\rm (a)]
\item
There exists a closed walk $W$ on $T$ such that
\begin{enumerate}[\rm (i)]
\item $W$ uses each edge of a circle block once and each edge of a block path twice, and
\item whenever $W$ meets a cut-vertex, it crosses from one block to another block.
\end{enumerate}
Moreover, each such closed walk $W$ is a minimal tour on $T$ with length $\ell(W)=\ell(T)$.
\item
Each minimal tour $W$ on $T$ satisfies the conditions {\rm(i)}, {\rm(ii)}, and $\ell(W)=\ell(T)$.
\item
There exists a unique direction $\omega_T$ of $T$ (up to negation of $\omega_T$) such that $(W,\omega_T)$ is coherent for each minimal tour $W$ on $T$.  The direction $\omega_T$ satisfies
\[
f_{(W,\,\omega_T)}=[\omega,\omega_T]\, I_T.
\]
\item
If $v$ is a cut-vertex of $T$ and $W=W_1W_2$,
where $W_1,W_2$ are closed sub-walks having initial and terminal vertices at $v$, then
both $W_1$ and $W_2$ are negative and both $(W_1,\omega_T)$ and $(W_2,\omega_T)$ are
incoherent at $v$ and coherent elsewhere.
\end{enumerate}
\end{prop}
\begin{proof}
(a) We prove it by the covering graph method, rather than working
directly on $T$ as in \cite{Chen-Wang3}. Let $\widetilde T$ be the
double covering graph of $T$. We claim that there exists a directed
circle $(\tW,\omega_{\tW})$ such that (i$'$) $\tW$ covers each edge
of a circle block once and each edge of a block path twice, (ii$'$)
the orientations on each edge of $T$ induced from
$(\tW,\omega_{\tW})$ by the projection are identical, and (iii$'$)
the induced orientation $\omega_T$ from $\omega_{\tW}$ is a
direction of $T$.

If $T$ is a circle block $C$, the circle $C$ must be positive, since there is no cut-vertex.  Choose a direction $\omega_C$ of $C$.  By Lemma~\ref{L:closedwalklift}, $(C,\omega_C)$ lifts to a directed closed walk $(\widetilde{C},\widetilde{\omega}_C)$ in $\widetilde T$ that covers $C$ once.
The closed walk $\widetilde{C}$ must be a circle since a self-intersection of $\widetilde{C}$ implies a self-intersection of the circle $C$, which is impossible.

Now we describe how to lift a circle-tree $T$ that contains at least two circle blocks to a circle $\tW$ in $\widetilde T$.

In a circle block $C=v_0x_1\cdots x_n v_n$ ($v_n=v_0$), call an
\emph{arc} a part of the circle that connects consecutive
cut-vertices on $C$ and let $p$ be the number of cut-vertices in
$C$, which is also the number of arcs.  There are two ways to lift
each arc, which are determined by the lift of its initial vertex
$v$ to either $v^+$ or $v^-$.  If we lift an arc $A=v_0x_1\cdots
x_k v_k$ so that $v_0$ lifts to $v_0^\alpha$, then $v_k$ lifts to
$v_k^{\alpha\,\sigma(A)}$.  For the next arc, $A'=v_kx_{k+1}\cdots
x_\ell v_\ell$, the initial vertex $v_k$ lifts to
$v_k^{-\alpha\,\sigma(A)}$, and the terminal vertex $v_l$ lifts to
$v_l^{-\alpha\sigma(AA')}$. Thus $\widetilde A$ jumps to $\widetilde
A'$ since they have no vertex in common. We continue this process
through all the arcs of $C$. The last lifted arc $\widetilde A''$
ends at vertex $v_0^{(-1)^{p-1}\alpha\,\sigma(C)}$ because there are
$p-1$ jumps from the first arc to the last of the $p$ arcs in $C$;
since $\sigma(C)=(-1)^p$, $\widetilde A''$ ends at $v_0^{-\alpha}$,
leaving a jump to the initial vertex $v_0^\alpha$ of $\widetilde A$.
For the lift of $C$ to use in $\tW$ we arbitrarily choose one of the
two possible lifts described by this rule.

A block path $P$ has two lifts that are vertex-disjoint paths
$\widetilde P, \widetilde P^*$. We use both of them in $\tW$.

We now prove $\tW$ is a circle. If there exists any sesqui-Eulerian
circle-tree for which $\tW$ is not a circle, let $T$ be one with the
fewest edges. We say two circle blocks are
\emph{attached} if they have a common vertex. A \emph{stout block}
$B$ is a circle block $C_B$ of $T$ together with the
(necessarily negative) loops that are attached to it. An \emph{end stout
block} is a stout block such that $C_B$ is an end block in $T$ with
loops (other than $C_B$) deleted. The important fact about a stout block $B$ is that
each attached loop, when lifted to $\widetilde T$, connects two
consecutive lifted arcs of $C_B$ into a path.  It follows that $\tW$
is a circle if $T$ has only one stout block.  Thus we may assume $T$ 
has more than one stout block. Let $B$ be an end stout block, joined by
a block path $P$ to another stout block $B_1$ at the vertex $u$ of
$B_1$.  The lift of $B \cup P$ is a path in $\widetilde T$
connecting $u^+$ to $u^-$. Replace $B \cup P$ by a single negative
loop at $u$, forming a new circle-tree $T'$ with fewer edges, whose
lifted walk $\tW'$ is therefore a circle. Thus $\tW$ is a circle.

The circle $\tW$ covers each edge of a circle block once
and each edge of a block path twice. We give
it a direction $\omega_{\tW}$. Let $e$ be an edge of a block path
and let $\End(e)=\{u,v\}$; then $\tW$ contains
$e^+=u^+v^{\sigma(e)}$ and $e^-=u^-v^{-\sigma(e)}$.  Then we may
write $\tW=\tW_1u^+e^+v^{\sigma(e)}\tW_2v^{-\sigma(e)}e^-u^-$. It
follows that $\omega_{\tW}(u^+,e^+)=-\omega_{\tW}(u^+,\tilde x)$ (by
coherence of $\omega_{\tW}$) $=\omega_{\tW}(u^-,\tilde y)$ (by Lemma
\ref{L:walksign} and positivity of the edges in $\widetilde\Sigma$)
$=-\omega_{\tW}(u^-,e^-)$, where $\tilde x$ is the first edge in
$\tW_1$ and $\tilde y$ is its last edge. This implies that the
projected orientations on $e$ from $e^+$ and $e^-$ are identical.

We have obtained a directed circle $(\tW,\omega_{\tW})$ in
$\widetilde T$ that satisfies the conditions (i$'$)--(iii$'$).
We define the closed walk $W$ by $W=\pi(\tW)$ and its
direction $\omega_W$ by $\omega_W=\pi(\omega_{\tW})$. Since
$\pi$ is locally an incidence-preserving bijection, it is clear that $W$ satisfies
conditions (i) and (ii) and is a minimal tour on $T$.

(b) A minimal tour $W$ on $T$ must traverse each edge in a block
path at least twice, since each such edge is a cut-edge.
So $W$ has length at least $\ell(T)$. Since $W$ has minimum length,
it follows that $\ell(W)=\ell(T)$. The minimality of $\ell(W)$
obliges $W$ to satisfy conditions (i) and (ii).

(c) Let $\omega_T$ be the direction of $T$ obtained by projecting
the direction $\omega_{\tW}$. For an arbitrary
minimal tour $W$ on $T$, the fact that $\omega_T$ is a direction
forces $(W,\omega_W)$ to be a directed closed walk. The uniqueness
of $\omega_T$ up to sign follows from the tree-like structure of $T$.
Properties (i) and (ii) imply that $I_T$ is a flow on
$(T,\omega_T)$. Thus $f_{(W,\,\omega_T)} =[\omega,\omega_T]\, I_T$.

(d) To prove negativity of $W_1$, consider that each block path implies two minus signs in $\sigma(W)$, one for each circle block with which it is incident.  Hence $\sigma(W)=+$ and $\sigma(W_1)=\sigma(W_2)$.  If the block path contains the cut-vertex that separates $W_1$ and $W_2$, then one incident circle is in $W_1$ and contributes a minus sign to $\sigma(W_1)$ while the other incident circle is in $W_2$ and does not contribute to $\sigma(W_1)$.  No other block path affects the sign of $W_1$.
\end{proof}

The proof shows that the number of directed minimal tours
on a sesqui-Eulerian circle-tree $T$
(up to changing the initial vertex and reversing the direction)
is $2^{q}$, where $q$ is the number of circle blocks in $T$.

\begin{lem}[Minimality of Sesqui-Eulerian Circle-Trees]\label{Minimality-of-Prime-Sesqui-Eulerian-Signed-Graph}
No directed sesqui-Eulerian circle-tree properly contains another directed sesqui-Eulerian circle-tree.
\end{lem}
\begin{proof}
Let $(T,\omega_T)$ be a directed sesqui-Eulerian circle-tree that has a proper subgraph $T'$ which is also a sesqui-Eulerian circle-tree with the direction $\omega_T|_{T'}$, the restriction of $\omega_T$ to $T'$.
The circle blocks of $T'$ are certainly circle blocks of $T$.
There exist a circle block $C$ of $T'$ and a vertex $u$ of $C$ such that $u$ is not a cut-vertex in $T'$ but is a cut-vertex in $T$.  It follows that $(C,\omega_T)$ must be coherent at $u$ when $C$ is considered as a circle block in $T'$ but must be incoherent (be either a sink or a source) at $u$ when $C$ is considered as a circle block in $T$. This is a contradiction.
\end{proof}

\begin{lem}\label{Lift-Flow}
Every flow $f$ on $(\Sigma,\omega)$ can be lifted to a flow
$\tilde f$ on $(\widetilde\Sigma,\widetilde\omega)$.
\end{lem}
\begin{proof}
This follows from (b) of Corollary~\ref{flow-walk} and Part (a) of Proposition~\ref{P:flowlift}.
\end{proof}

\begin{thm}[Resolution of Indecomposable Flows]\label{Lift:Indecomposable-Flow}
If $f$ is a conformally indecomposable flow on
$(\Sigma,\omega)$, then there exists a directed closed, positive walk
$(W,\omega_f)$ on $\Sigma(f)$ such that $f=f_{(W,\,\omega_f)}$.

Furthermore, if $(\tW,\,\widetilde{\omega}_f)$ is a lift of
$(W,\,\omega_f)$, then  $\tW$ is a circle, and
$f_{(\tW,\,\widetilde{\omega}_f)}$ is a lift of $f$ and is a conformally
indecomposable flow on $(\widetilde\Sigma,\widetilde\omega)$.
\end{thm}
\begin{proof}
The conformal indecomposability of $f$ implies that $\Sigma(f)$ is
connected. By (b) of Corollary~\ref{flow-walk}, there exists a
directed closed, positive walk $(W,\omega_f)$ on $\Sigma(f)$ such
that $f=f_{(W,\,\omega_f)}$. Lift $(W,\omega_f)$ to a directed
closed walk $(\tW,\widetilde{\omega}_f)$ in $\widetilde\Sigma$. Now
$f$ is lifted to a flow $f_{(\tW,\,\widetilde{\omega}_f)}$ of
$(\widetilde\Sigma,\widetilde\omega)$ by
Proposition~\ref{P:flowlift}(a), denoted by
$\tilde f=f_{(\tW,\,\widetilde\omega_f)}$.

Suppose $\tilde f$ is decomposed as $\tilde f=\tilde f_1+\tilde
f_2$, where $\tilde f_i$ are nonzero integral flows and $\tilde
f_1\, \tilde f_2\geq 0$. Notice that
\[
f=[\omega,\omega_f]\,|f|, \quad \tilde f
=[\widetilde\omega,\widetilde{\omega}_f]\, |\tilde f|,
\quad
\tilde f_i=[\widetilde\omega,\widetilde{\omega}_f]\, |\tilde f_i|.
\]
Let $f_i=\pi(\tilde f_i)$, which are nonzero flows on
$(\Sigma,\omega)$. For each edge $e$ of $\Sigma$,
\begin{eqnarray*}
f_i(e) &=& [\tilde \omega,\widetilde{\omega}_f](e^+)\,|\tilde f_i|(e^+)
+ [\tilde \omega,\widetilde{\omega}_f](e^-)\,|\tilde f_i|(e^-) \\
&=& [\omega,\omega_f](e)\big(|\tilde f_i(e^+)|+|\tilde f_i(e^-)|\big) \\
&=& [\omega,\omega_f](e)\,\pi(|\tilde f_i|)(e).
\end{eqnarray*}
Since $\pi(|\tilde f_i|)\geq 0$, it follows that $f_i = [\omega,\omega_f]\, |f_i|$.
Thus $f=f_1+f_2$ by linearity of $\pi$, and $f_1\, f_2=|f_1|\, |f_2|\geq 0$, meaning that
$f$ is conformally decomposable. This is a contradiction.

The conformal indecomposability of $f_{(\tW,\,\widetilde{\omega}_f)}$ implies that it is a graphic circuit flow. So $\tW$ is a circle.
\end{proof}

\noindent {\bf Remark.} The projection of a flow is a flow, but the
projection of a conformally indecomposable flow is not necessarily a
conformally indecomposable flow.

\begin{thm}[Classification of Conformally Indecomposable Flows] \label{T:flows}
Let $f$ be a flow on $(\Sigma,\omega)$. Then $f$ is conformally
indecomposable if and only if $\Sigma(f)$ is a sesqui-Eulerian circle-tree
with a direction $\omega_f$ and
\begin{equation*}
f=[\omega,\omega_f]\, I_{\Sigma(f)}.
\end{equation*}
\end{thm}
\begin{proof}
$\Rightarrow$: We assume $f$ is conformally indecomposable.
Recall the directed closed, positive walks
$(W,\omega_f)$ on $\Sigma(f)$ and $(\tW,\widetilde{\omega}_f)$ on $\widetilde{\Sigma}$ in
Theorem~\ref{Lift:Indecomposable-Flow}, where $\tW$ is a lift
of $W$. Note that $f=f_{(W,\,\omega_f)}$, $|f|=[\omega,\omega_f]f$, and $\tW$ is a circle. Let $W=v_0e_1v_1\cdots e_nv_n$ and
$\tW=v_0^{\alpha_0}\tilde{e}_1 v_1^{\alpha_1} \tilde{e}_2\cdots \tilde{e}_n v_n^{\alpha_n}$.
Since $\widetilde\Sigma$ is a double covering of $\Sigma$ and $\pi(\tW)=W$, it follows
that vertices and edges appear in $W$ at most twice.

If $W$ has no double vertices, that is, $W$ has no self-intersections, then
$(W,\omega_f)$ is a directed circle and $W$ is a positive circle
since it is a one-to-one image of a circle in $\widetilde\Sigma$.
Then $W$ is a sesqui-Eulerian circle-tree and $f=[\omega,\omega_f]\, I_{\Sigma(f)}$.

Assume that $W$ has some self-intersections. Let $v$ be a double
vertex of $W$. Rewrite $W$ as $W_1W_2$, where each $W_i$ is a closed
walk with the initial and terminal vertices at $v$. More
specifically, $W_1=v_0e_1v_1\cdots e_mv_m$ and
$W_2=v_me_{m+1}v_{m+1}\cdots e_nv_n$, where $v_0=v_m=v_n=v$. We
claim that each $W_i$ is negative, each $(W_i,\omega_f)$ is
incoherent at $v$ and coherent elsewhere, and $v$ is a
cut-vertex of $\Sigma(W)$.

Each $(W_i,\omega_f)$ is coherent everywhere
except at $v$. Suppose $(W_1,\omega_f)$ is coherent at $v$.
That forces $(W_2,\omega_f)$ to also be coherent at $v$. Then each
$(W_i,\omega_f)$ is a directed closed, positive walk. We thus have
$f_W=f_{W_1}+f_{W_2}$ within $(\Sigma,\omega_f)$, meaning that
$|f|$ is conformally decomposable; this is a contradiction by
Proposition~\ref{L:Indecomposability-Minimality}. Hence
$(W_i,\omega_f)$ must be incoherent at $v$.
Lemma~\ref{L:walkflow} implies that the closed walks $W_i$ are
negative. Write $\tW$ as $\tW_1\tW_2$, where
\[
\tW_1=v_0^{\alpha_0} \tilde{e}_1 v_1^{\alpha_1}\cdots \tilde{e}_m v_m^{\alpha_m}
\quad\text{and}\quad
\tW_2=v_m^{\alpha_m} \tilde{e}_{m+1} v_{m+1}^{\alpha_{m+1}}\cdots \tilde{e}_n v_n^{\alpha_n}
\]
are paths. We have $\alpha_0=\alpha_n=-\alpha_m$ by
Lemma~\ref{L:closedwalklift}.

Suppose $v$ is not a cut-vertex, that is, $W_1$ and $W_2$ meet at a
vertex $u$ other than $v$. Let $u$ occur as $v_k$ in $W_1$ and as
$v_h$ in $W_2$, that is, $u=v_k=v_h$, where $1\leq k\leq m-1$ and
$m+1\leq h\leq n-1$. Since $\tW$ is a circle, we have
$v_k^{\alpha_k}\neq v_h^{\alpha_h}$, consequently,
$\alpha_k=-\alpha_h$. Consider the closed walk
\[
\tW'=v_0^{\alpha_0} \tilde{e}_1 v_1^{\alpha_1}\cdots v_{k-1}^{\alpha_{k-1}}\tilde{e}_k
v_k^{\alpha_k}\tilde{e}_h^*v_{h-1}^{-\alpha_{h-1}}
\cdots v_{m+1}^{-\alpha_{m+1}}\tilde{e}_{m+1}^* v_m^{-\alpha_m}
= \tW'_1\tW'_2,
\]
where $v_k^{\alpha_k}=v_h^{-\alpha_h}$, $v_m^{-\alpha_m}=v_0^{\alpha_0}$, and
\[
\tW'_1=v_0^{\alpha_0} \tilde{e}_1 v_1^{\alpha_1}\cdots v_{k-1}^{\alpha_{k-1}}\tilde{e}_k
v_k^{\alpha_k}
\quad\text{and}\quad
\tW'_2=v_h^{-\alpha_h}\tilde{e}_h^*v_{h-1}^{-\alpha_{h-1}}
\cdots v_{m+1}^{-\alpha_{m+1}}\tilde{e}_{m+1}^* v_m^{-\alpha_m}.
\]
Let $s=\widetilde{\omega}_f(v^{\alpha_0}_0,\tilde{e}_1)={\alpha_0}{\omega}_f(v_0,{e}_1)$,
where the second equality follows
from the definition \eqref{Lift-Orientation}. Now
$\widetilde{\omega}_f(v^{\alpha_k}_k,\tilde{e}_k)=-s$, since
as an open walk $(\tW_1,\widetilde{\omega}_f)$ is directed.
Similarly,
$s=\widetilde{\omega}_f(v^{\alpha_m}_m,\tilde{e}_{m+1})=\alpha_m\omega_f(v_m,e_{m+1})$.
Hence $\widetilde{\omega}_f(v^{-\alpha_m}_m,\tilde{e}_{m+1}^*)
=-\alpha_m\omega_f(v_m,e_{m+1})=-s$, that is,
$(\tW',\widetilde{\omega}_f)$ is coherent at
$v_0^{\alpha_0}$. Thus $(\tW',\widetilde{\omega}_f)$ is a
directed closed walk on $\widetilde{\Sigma}(f)$.
Let $W'=\pi(\tW')$. It follows from
Lemma~\ref{L:closedwalklift} that $(W',\omega_f)$ is a directed
closed, positive walk on $\Sigma(f)$. Moreover, for a repeated
edge $e$ in $W'$: if $e$ is the projection of two edges in
$\tW'_1$ or two in $\tW'_2$, then $e$ is the projection of two edges in $\tW_1$ or two in $\tW_2$;
if $e$ is the projection of one edge in $\tW'_1$ and another in $\tW'_2$, then $e$ is
the projection of one edge in $\tW_1$ and one edge in $\tW_2$;
it follows that $e$ is a repeated edge in $W$.
Thus $E(W')$ is a proper sub-multiset of $E(W)$ with the same edge orientations. Consequently
$f_{W'}\leq f_W$ and $f_{W'}\neq f_W$. This is contradictory to the minimality of $f_W$.
It follows that a repeated vertex in $W$ is a cut-vertex of $\Sigma(W)$.

The subgraph $\Sigma(f)$ is obtained from the circle $\tW$ by projection, which identifies
each pair $v^+,v^-$ of vertices and each pair $\tilde e, \tilde e^*$ of edges of which both appear in $\tW$.
Since each identified vertex $v$ is a cut-vertex of $\Sigma(f)$, each identified edge $e$ is a cut-edge.
Removing all cut-edges from $\Sigma(f)$, every vertex has degree 2 or 4; in the latter case it is a cut-vertex that separates two of its incident edges from the other two.  It follows that in $\Sigma(f)$ with cut edges removed, every block is a circle.  The connected components of the identified vertices and
edges form block paths $P_j$ (possibly of zero length) joining
some pairs of the circles $C_i$.  Thus $\Sigma(f)$ satisfies
(a, b, d) of Definition \ref{cycle-tree-defn}.  It satisfies (c) because an end block that is an edge would have a vertex of degree 1 in $\Sigma(f)$, which is impossible.

Recall the incoherence of $(W_i,\omega_f)$ at the double
vertex $v$. It follows that each $(C_i,\omega_f)$ is
incoherent at the cut-vertices of $\Sigma(f)$ on $C_i$ and coherent
elsewhere. Thus the circle $C_i$ has sign $(-1)^p$ by
Lemma~\ref{L:walksign}, where $p$ is the number of cut-vertices on
$C_i$.

We have proved that $\Sigma(f)$ is a sesqui-Eulerian
circle-tree, $\omega_f$ is a direction of $\Sigma(f)$, and
$(W,\omega_f)$ is a directed minimal tour. Under these conditions, each
edge in a circle block appears once in $W$ and each
edge in a block path appears twice in $W$, so
$f_W=I_{\Sigma(f)}$ within $(\Sigma,\omega_f)$. Therefore
$f = f_{(W,\,\omega_f)} = [\omega,\omega_f]\,I_{\Sigma(f)}$.

$\Leftarrow$: We assume $\Sigma(f)$ is a sesqui-Eulerian circle-tree with direction $\omega_f$ and $f = [\omega,\omega_f]\,I_{\Sigma(f)}$.
Let us write $f=\sum_{i=1}^k f_i$, where the $f_i$ are conformally
indecomposable flows on $(\Sigma,\omega)$ that conform to $f$, and $k\geq1$.  Let
$(\Sigma_i,\omega_i)$ be sesqui-Eulerian circle-trees such that
$f_i= [\omega,\omega_i] \, I_{\Sigma_i}$. Since each $f_i$
conforms to the sign pattern of $f$, it follows that $\Sigma_i$
is a subgraph of $\Sigma(f)$ and $\omega_{f_i}$ is the restriction of $\omega_f$ to $\Sigma_i$.
The circle blocks of $\Sigma_i$ are
certainly circle blocks of $\Sigma(f)$. Consider a block
path $P$ of $\Sigma_i$. If $\ell(P)=0$, then $P$ is the intersection
of two circle blocks $C_1$ and $C_2$ of $\Sigma_i$.
Since $C_1, C_2 \subseteq \Sigma(f)$, $P$ is a block path in $\Sigma(f)$. If
$\ell(P)>0$, then $f_i(e)=\pm 2$ for all edges $e$ of $P$. Since
$f_i$ conforms to the sign pattern of $f$ and $|f_i|\leq|f|=I_{\Sigma(f)}$, it must be that
$f(e)=\pm 2$ for every edge $e$ of $P$.
This means that $P$ is a block path of $\Sigma(f)$.
Hence $(\Sigma_i,\omega_i)$ is a directed sesqui-Eulerian circle-tree contained in $(\Sigma(f),\omega_f)$.
It follows by Lemma~\ref{Minimality-of-Prime-Sesqui-Eulerian-Signed-Graph} that $\Sigma(f)=\Sigma_i$.
Therefore $k=1$, i.e., $f$ is conformally indecomposable.
\end{proof}

\begin{prop}\label{prime}
A signed graph $\Sigma$ is prime sesqui-Eulerian if and only if $\Sigma$ is a sesqui-Eulerian circle-tree.
\end{prop}
\begin{proof}
If $\Sigma$ is a sesqui-Eulerian circle-tree, then by Proposition~\ref{Prime-Sesqui-Eulerian-Property} there exist an orientation $\omega$ on $\Sigma$ and a closed walk $W$ that uses every edge of $\Sigma$ once or twice such that $(W,\omega)$ is a directed closed, positive walk.
If $(W,\omega)$ did properly contain a directed closed, positive subwalk $(W',\omega')$, then $(\Sigma, \omega)$ would properly contain a directed sesqui-Eulerian signed graph $(\Sigma',\omega')$, where $\Sigma'$ consists of the vertices and edges of $W'$; but that would contradict Lemma \ref{Minimality-of-Prime-Sesqui-Eulerian-Signed-Graph}.
Thus $\Sigma$ is a prime sesqui-Eulerian signed graph.

Conversely, if $\Sigma$ is a prime sesqui-Eulerian signed graph, then
by definition there exist an orientation $\omega$ on $\Sigma$ and a
closed positive walk $W$ that uses every edge of $\Sigma$ once or
twice, such that $(W,\omega)$ is a directed walk.  The flow
$f_W$ on $(\Sigma,\omega)$ can be conformally decomposed into conformally indecomposable flows $f_i$ so that $f_W=\sum_{i=1}^m f_i\geq 0$. By Theorem~\ref{T:flows}, for each $f_i$ there exist a sesqui-Eulerian circle-tree $T_i$ and its direction $\omega_{T_i}$ such that $f_i=[\omega,\omega_{T_i}]\,I_{T_i}$. Hence
$\omega_{T_i}=\omega$ on $T_i$. By Theorem~\ref{Prime-Sesqui-Eulerian-Property}, there exists a directed closed, positive walk $(W_i,\omega_{T_i})$ on each $T_i$. 
Since $\Sigma$ is connected, we can construct a directed closed positive walk $(W',\omega)$ that uses every edge of $\Sigma$ exactly as often as does $(W,\omega)$, by concatenating the $W_i$ in a suitable order after changing their initial/terminal vertices as necessary. 
Since $\Sigma$ is prime and $(W',\omega)$ contains the directed closed, positive subwalk $(W_1,\omega_{T_1})$, $\Sigma$ must be the sesqui-Eulerian circle-tree $T_1$ and $W'=W_1$.
\end{proof}

\begin{cor}\label{minimal}
A signed graph is a minimal sesqui-Eulerian signed graph if and only if it is a signed-graph circuit.
\end{cor}
\begin{proof}
Let $\Sigma$ be a minimal sesqui-Eulerian signed graph.  Since it is prime, it is a sesqui-Eulerian circle-tree by Proposition \ref{prime}.  If it has only one circle block, it is a positive circle, i.e., a circuit of Type I.  If it has more than one circle block, it has two (or more) end blocks $C_1, C_2$, each of which is a negative circle, and it contains a path $P$ (possibly of length 0) connecting those blocks.  Then $C_1 \cup P \cup C_2$ is a circuit of Type II or III.

Conversely, a signed-graph circuit is obviously a sesqui-Eulerian circle-tree and minimal.
\end{proof}

\begin{thm}[Half-Integer Decomposition]\label{Half-Integer Scale Decomposition}
Let $T$ be a sesqui-Eulerian circle-tree with a direction $\omega_T$.
Either $T$ is a signed-graph circuit, or there exists a closed, positive walk $W$ on $T$,
\[
W=P_1C_1P_2\cdots P_nC_{n}, \quad n\geq 2,
\]
satisfying the following conditions:
\begin{enumerate}[\rm (a)]
\item
The $C_i$, $1\leq i\leq n$, are the end blocks of $T$, and the $P_i$ are paths of positive length.
\item
Each edge of a non-end circle block appears in exactly one of the paths $P_i$, and each edge of a block path appears in exactly two of the paths $P_i$.
\item
Each $(C_{i}P_{i+1}C_{i+1},\omega_T)$, where $C_{n+1}=C_1$, is a directed circuit of Type III.
\item
$T$ is a conformal half-integral linear combination of signed-graph circuits; more precisely,
$$I_T = \frac{1}{2} \sum_{i=1}^{n} I_{T_i},$$
where $T_i$ is the restriction of $T$ to the subgraph $C_{i} \cup P_{i+1} \cup C_{i+1}$.
\end{enumerate}
\end{thm}
\begin{proof}
Let $\tW$ be the circle in $\widetilde{T}$ that covers $T$, constructed in the proof of Proposition \ref{Prime-Sesqui-Eulerian-Property}.
This circle has the form $\widetilde{P}_1\widetilde{C}_1\widetilde{P}_2\widetilde{C}_2\cdots\widetilde{P}_n\widetilde{C}_n$, where the $\widetilde{C}_i$ are paths that are lifts of the $n$ end blocks $C_1,C_2,\ldots,C_n$ of $T$ and the $\widetilde{P}_i$ are connecting paths, necessarily of positive length because if two end blocks had a common vertex, $T$ would be a circuit of Type II.
The projection of $\tW$ is
\[
W=P_1C_1P_2 \cdots C_n,
\]
where $P_i=\pi(\widetilde{P}_i)$ is a path covered once by $\widetilde{P}_i$ because $P_i$ connects two different end blocks of a sesqui-Eulerian circle-tree.  Each walk $C_{i}P_{i+1}C_{i+1}$ is coherent because $\tW$ is coherent.
Property (b) follows from Part (a)(i) of Proposition \ref{Prime-Sesqui-Eulerian-Property}.
Property (c) follows from coherence of $\tW$ and of its projection $W$, Part (c) of Proposition \ref{Prime-Sesqui-Eulerian-Property}.
Since each end block is a negative circle, each $C_{i} \cup P_{i+1} \cup C_{i+1}$ is a directed circuit of Type III.
Each edge of a non-end circle block
appears in one of the paths $P_i$, and each edge of a block
path appears in two of the paths $P_i$.  Hence $I_T=f_{W}$ and
\[
I_T=I_{C_0}+I_{P_{n+1}}+\sum_{i=1}^n \big(I_{C_i}+I_{P_i}\big)
= \frac{1}{2} \sum_{i=0}^n I_{T_i}.
\qedhere
\]
\end{proof}

For an example, the weights on the edges of the circle-tree in
Figure~\ref{Prime-Sesqui-Eulerian-Signed-Graph-Exmp} form,
with respect to the given direction, a conformally indecomposable flow which
is the characteristic vector of the circle-tree.  This
conformally indecomposable flow can be decomposed into one-half of
the sum of three signed-graph circuit flows, as demonstrated in
Figure~\ref{Figg-Half-Integer-Decom}.

\begin{figure}[h]
\centering
\subfigure{\includegraphics[height=20mm]{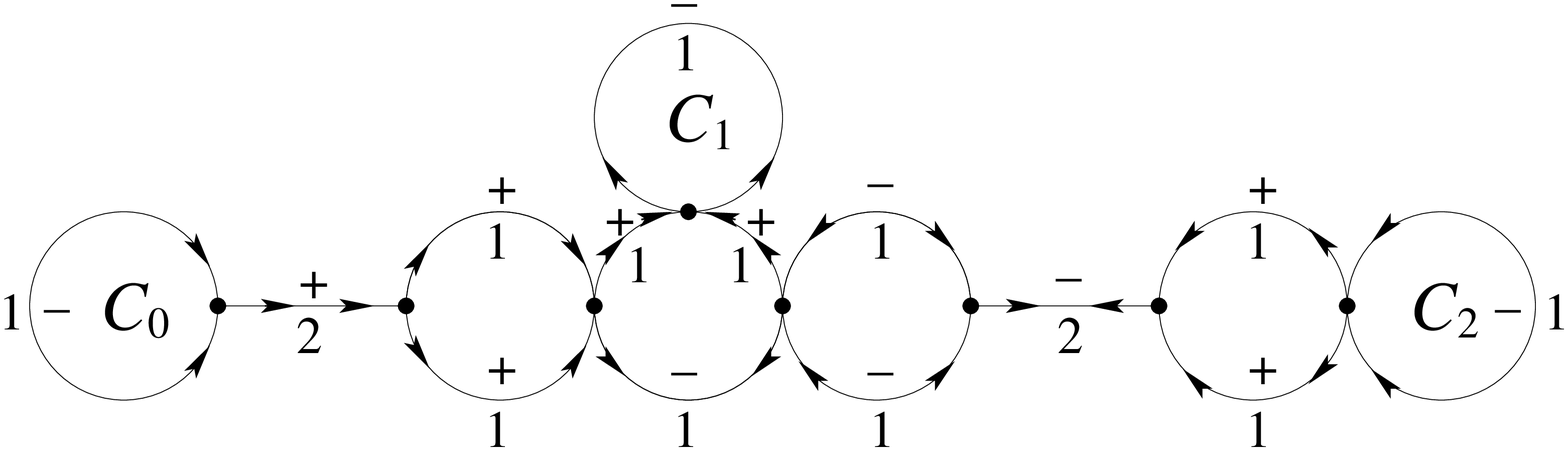}}
$\begin{array}{c} =\\ \\ \\
\end{array}$
\subfigure{\includegraphics[height=14mm]{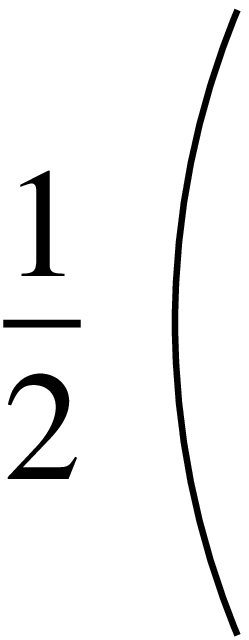}}
\hspace{-2mm}
\subfigure{\includegraphics[height=20mm]{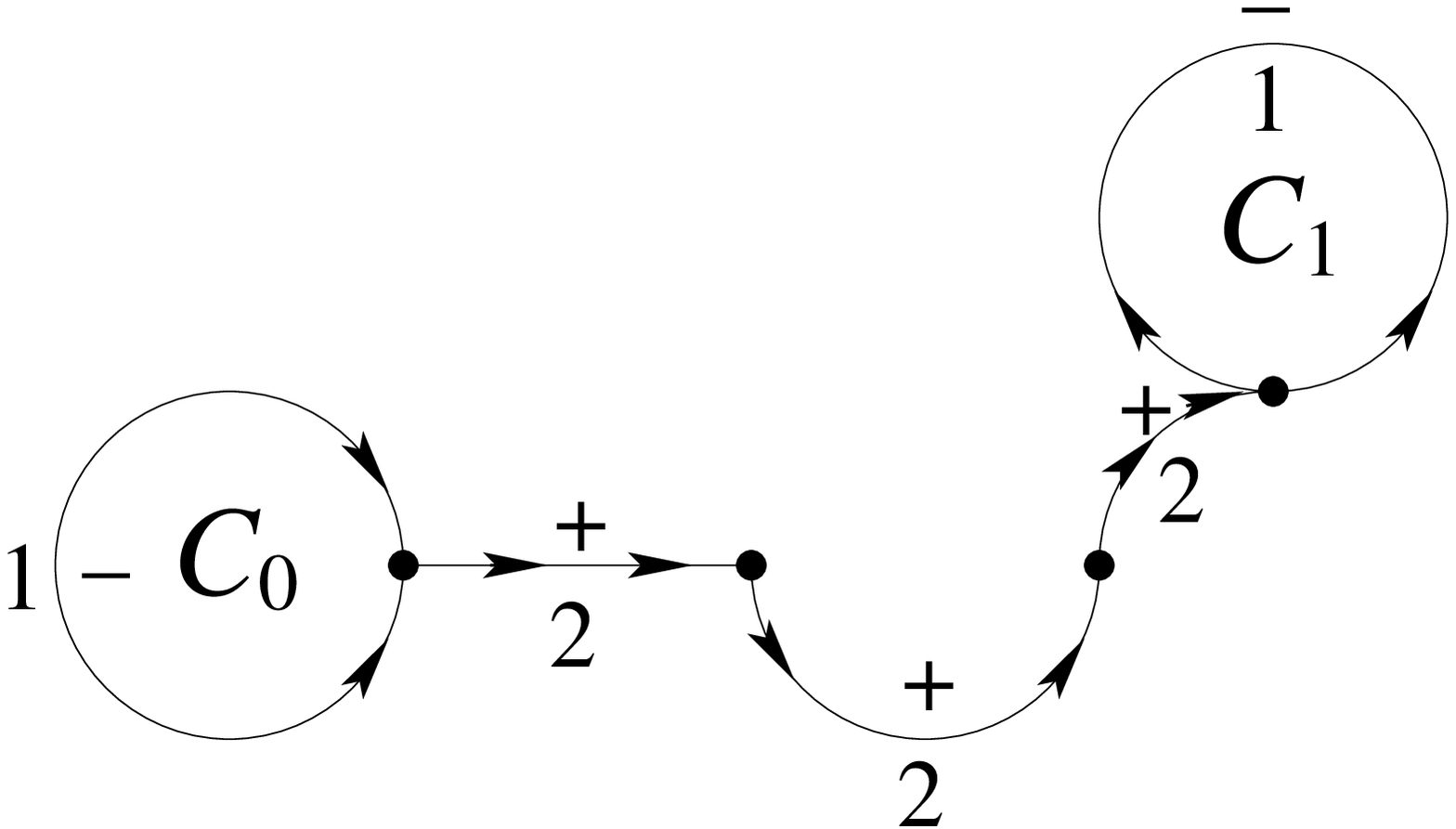}}
\hspace{0mm} \subfigure{\includegraphics[height=7mm]{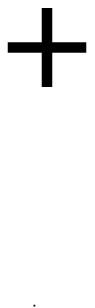}}
\hspace{0mm}
\subfigure{\includegraphics[height=20mm]{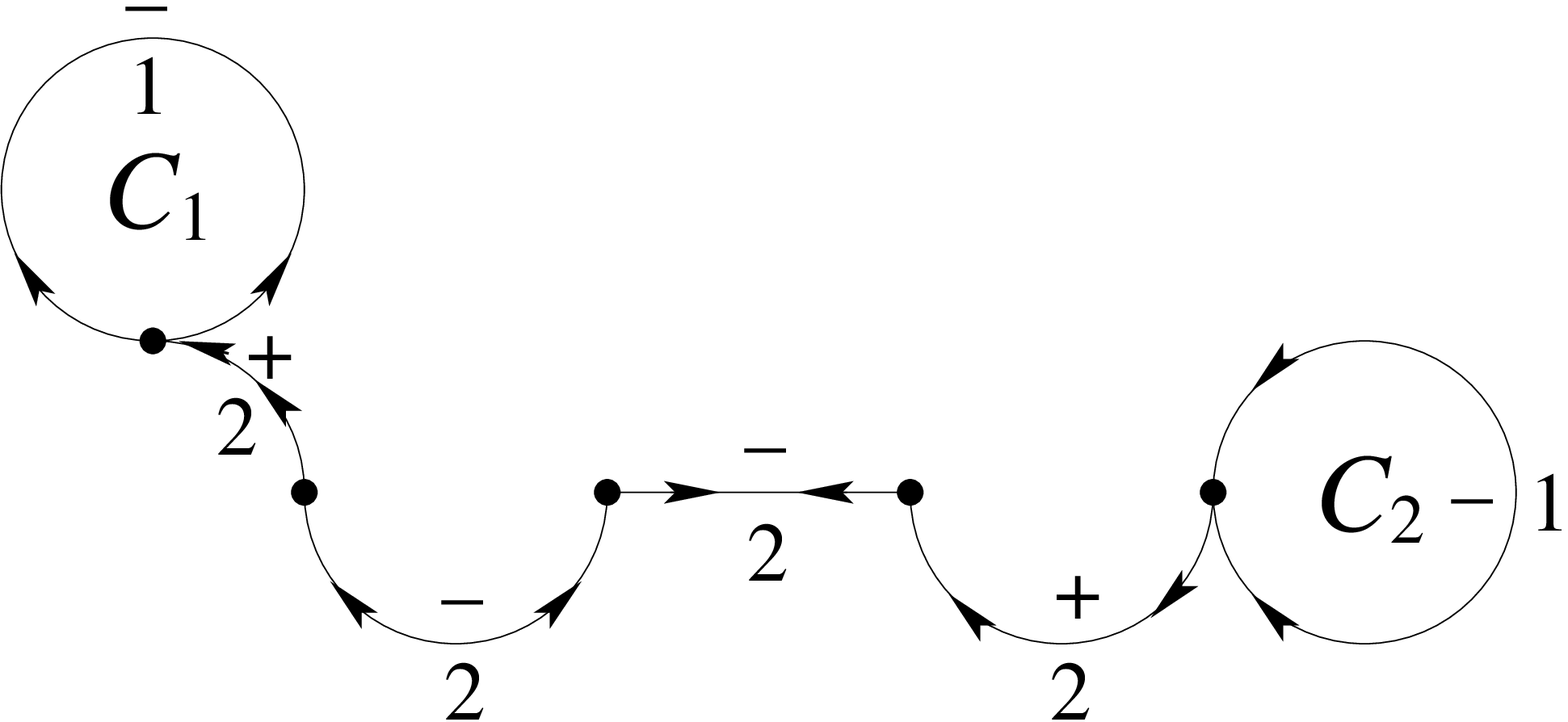}}
\hspace{0mm} \subfigure{\includegraphics[height=7mm]{plus-sign.eps}}
\hspace{0mm}
\subfigure{\includegraphics[height=12mm]{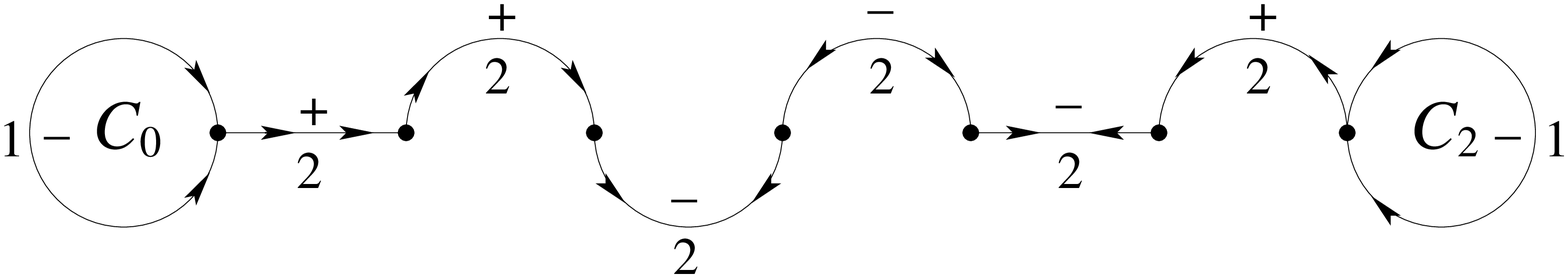}}
\subfigure{\includegraphics[height=14mm]{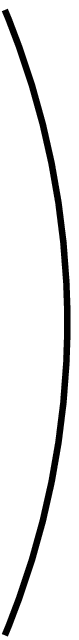}}
\caption{A conformally indecomposable flow is decomposed conformally into halves of signed-graph circuit flows of Type III.}
\label{Figg-Half-Integer-Decom}
\end{figure}

The half-integrality phenomenon in Theorem~\ref{Half-Integer Scale Decomposition}(d)
has appeared previously in connection with flows on
signed graphs (though not so named), possibly first in work of
Bolker (e.g., \cite{B2}) and later in \cite{Hoch,AK,HIB}
(see \cite{HIB} for references and explanation); and
also in \cite[Corollary 1.4, p.~283]{Geelen-Guenin},
which concerns a completely different problem.
The phenomenon in both cases is a consequence of the signs on the edges.

\begin{figure}[h]
\includegraphics[width=70mm]{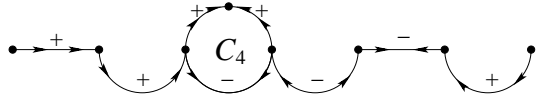}
\caption{A maximal independent set of the sesqui-Eulerian circle-tree in Fig.~\ref{Prime-Sesqui-Eulerian-Signed-Graph-Exmp}.}
\label{Max-independ-set}
\end{figure}

One may consider decomposition of integral flows without conforming sign patterns. It is clear from the definition of conformal decomposition that every nonzero integral flow
is a positive integral linear combination of conformally
indecomposable flows.  It is also true that each conformally
indecomposable flow can be further decomposed into an integral linear
combination of signed-graph circuit flows without conforming
sign patterns. Thus each integral flow on a signed graph is an
integral linear combination of circuit flows.
This fact is already explicitly given in terms of a maximal
independent edge set (that is, a matroid basis) in \cite[Eq.~(4.7) of Theorem~4.9, p.~275]{Chen-Wang1}.
For instance, the signed graph in
Figure~6 is a maximal independent set of the sesqui-Eulerian circle-tree in
Figure~\ref{Prime-Sesqui-Eulerian-Signed-Graph-Exmp}. The conformally
indecomposable flow in Figure~\ref{Prime-Sesqui-Eulerian-Signed-Graph-Exmp}
is further decomposed into circuit flows in Figure~7
without conforming signs.  We summarize these
observations in the following Corollary~\ref{Pos-decom-integer-decom}.

\begin{figure}[h]
\centering
\subfigure{\includegraphics[height=20mm]{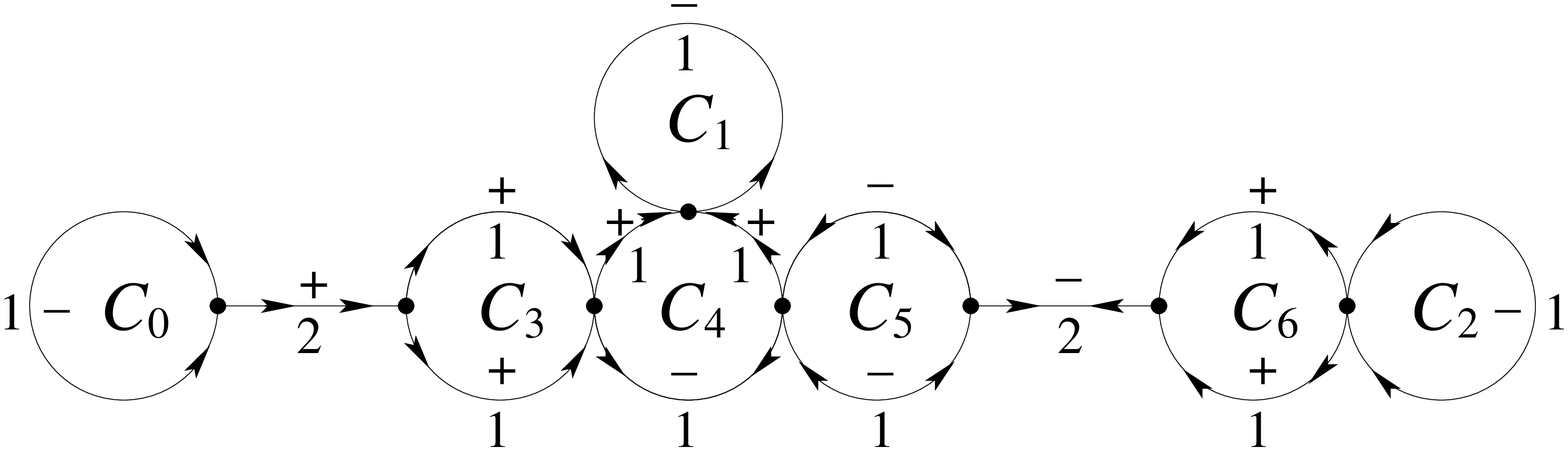}}
\hspace{0mm}$\begin{array}{c} =\\ \\ \\
\end{array}$
\hspace{-1mm}
\subfigure{\includegraphics[height=11mm]{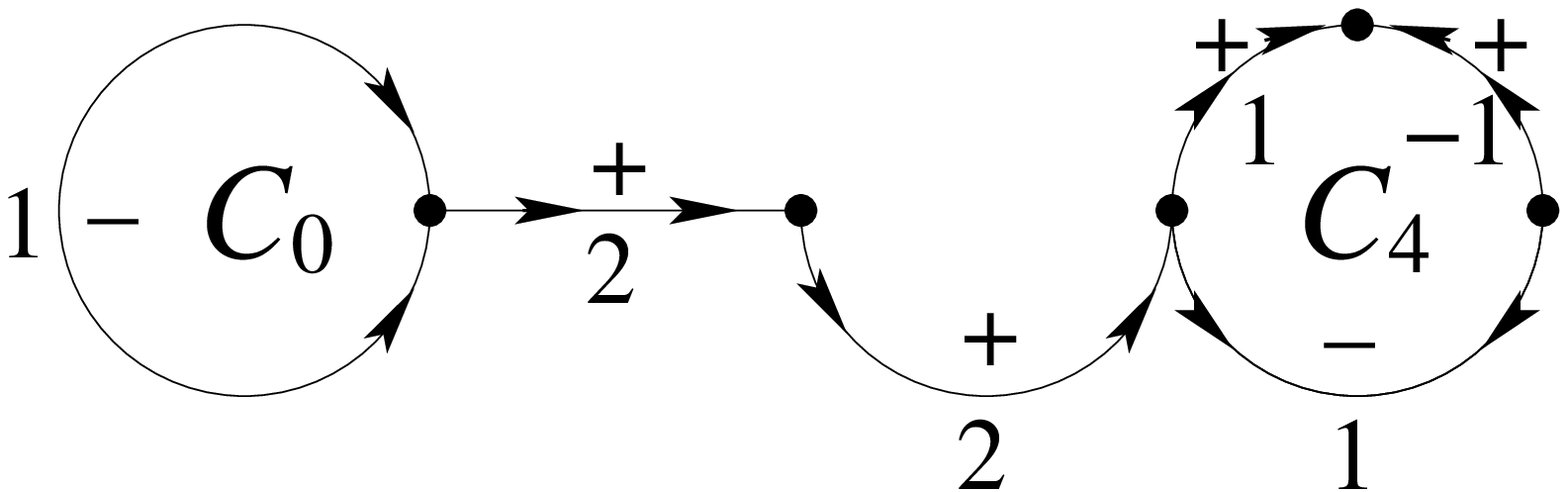}}
\hspace{0mm} \subfigure{\includegraphics[height=7mm]{plus-sign.eps}}
\hspace{0mm}
\subfigure{\includegraphics[height=20mm]{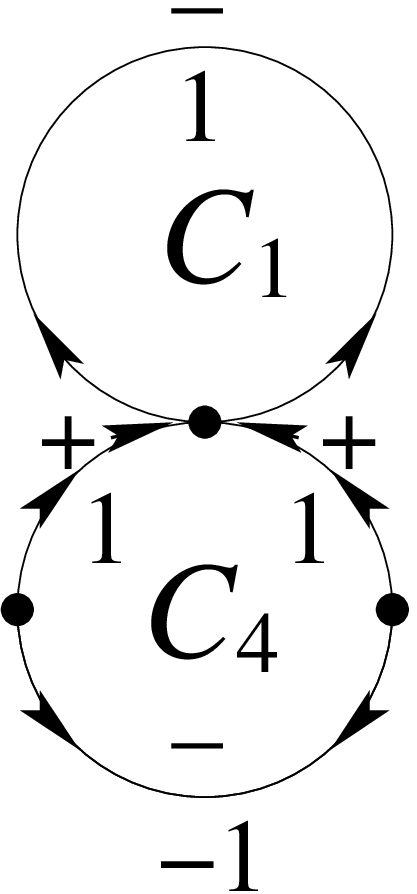}}\\
\hspace{0mm} \subfigure{\includegraphics[height=7mm]{plus-sign.eps}}
\hspace{0mm}
\subfigure{\includegraphics[height=11mm]{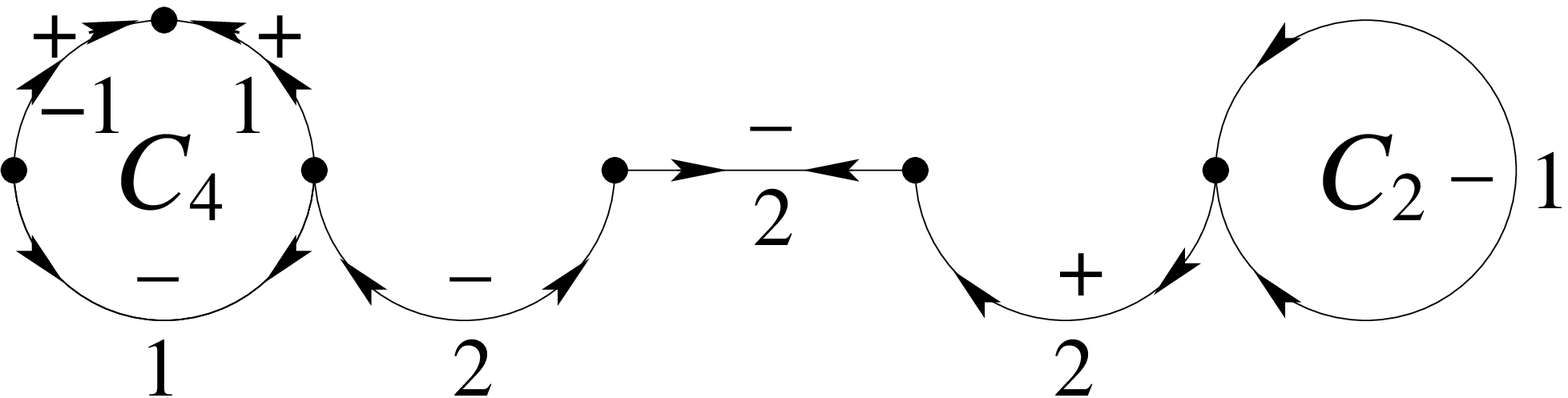}}
\hspace{0mm} \subfigure{\includegraphics[height=7mm]{plus-sign.eps}}
\hspace{0mm}
\subfigure{\includegraphics[height=12mm]{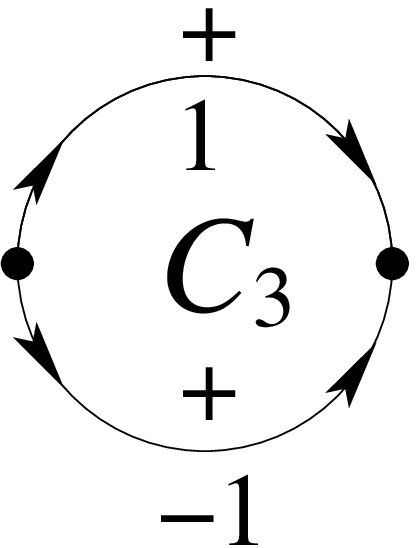}}
\hspace{0mm} \subfigure{\includegraphics[height=7mm]{plus-sign.eps}}
\hspace{0mm}
\subfigure{\includegraphics[height=12mm]{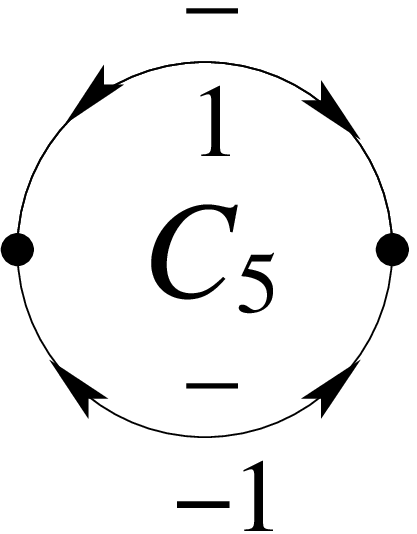}}
\hspace{0mm} \subfigure{\includegraphics[height=7mm]{plus-sign.eps}}
\hspace{0mm}
\subfigure{\includegraphics[height=12mm]{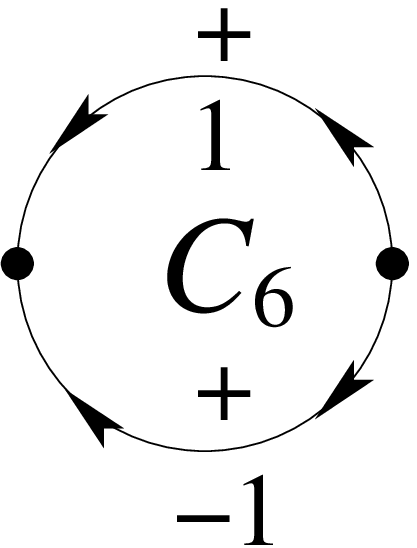}}
\caption{A conformally indecomposable flow is decomposed non-conformally but integrally into circuit flows.}
\label{AAAA}
\end{figure}

\begin{cor}\label{Pos-decom-integer-decom}
\begin{enumerate}[\rm (a)]
\item
If $f$ is a nonzero integral flow on a signed graph, then $2f$ can be conformally decomposed into a positive integral linear combination of signed-graph circuit flows.
\item
Every nonzero integral flow on a signed graph can be decomposed (conformally or non-conformally) into a positive integral linear combination of signed-graph circuit flows.
\end{enumerate}
\end{cor}

The following proposition is elementary but helps explain why there are
exactly three natural types of circuit for signed graphs, as introduced by
Zaslavsky \cite{SG}.

\begin{prop}
For a signed graph $\Sigma$, the following statements are equivalent.
\begin{enumerate}[\rm (a)]
\item $\Sigma$ is a minimal sesqui-Eulerian signed graph.
\item $\Sigma$ is a minimal prime sesqui-Eulerian signed graph.
\item $\Sigma$ is a minimal sesqui-Eulerian circle-tree.
\item $\Sigma$ is a signed-graph circuit.
\end{enumerate}
\end{prop}
\vspace{1ex}

\section{Acknowledgement}
The authors thank the two referees
and the editor for carefully reading the manuscript and offering valuable
suggestions and comments.


\end{document}